\documentclass[11pt]{article}

\usepackage[margin=1in]{geometry}
\usepackage{amsmath,amssymb,amsthm,mathtools}
\usepackage{array,booktabs,tabularx}
\usepackage{microtype}
\usepackage[authoryear,round]{natbib}
\usepackage[hidelinks]{hyperref}
\hypersetup{
  pdftitle={On the Log Determinant of Sample Correlation Matrices under Gaussianity},
  pdfauthor={Hongru Zhao}
}

\newcolumntype{L}[1]{>{\raggedright\arraybackslash}p{#1}}
\newcolumntype{Y}{>{\raggedright\arraybackslash}X}

\theoremstyle{plain}
\newtheorem{theorem}{Theorem}[section]
\newtheorem{proposition}[theorem]{Proposition}
\newtheorem{lemma}[theorem]{Lemma}
\newtheorem{corollary}[theorem]{Corollary}
\theoremstyle{definition}

\newtheorem{remark}[theorem]{Remark}

\numberwithin{equation}{section}

\newcommand{\tr}{\operatorname{tr}}
\newcommand{\diag}{\operatorname{diag}}
\newcommand{\E}{\mathbb{E}}
\newcommand{\Var}{\operatorname{Var}}
\newcommand{\Cov}{\operatorname{Cov}}
\newcommand{\Log}{\operatorname{Log}}
\newcommand{\Rhat}{\widehat R}
\newcommand{\op}{\mathrm{op}}
\newcommand{\F}{\mathrm{F}}
\newcommand{\dd}{\,\mathrm d}
\newcommand{\1}{\mathbf 1}

\begin{document}

\title{On the Log Determinant of Sample Correlation Matrices under Gaussianity}
\author{Hongru Zhao\\
\small School of Statistics, University of Minnesota\\
\small 355 Ford Hall, 224 Church Street SE\\
\small Minneapolis, Minnesota 55455, United States\\
\small \texttt{zhao1118@umn.edu}}
\date{}
\maketitle

\begin{abstract}
We prove a central limit theorem for the log determinant of a Gaussian Pearson sample correlation matrix as the dimension diverges. Only two conditions are imposed: the population correlation matrix is positive definite, and the sample degrees of freedom are at least the dimension. Both are necessary for the ordinary log determinant to be finite. To the best of our knowledge, no previous central limit theorem covers this full nonsingular domain. It covers every aspect ratio from dilute growth to the square hard edge. No uniform lower or upper bound is imposed on the eigenvalues of the population correlation matrices: the smallest may approach zero and the largest may diverge. The proof develops a coordinatewise Wiener chaos reduction for the random diagonal normalization and combines it with an exact Wishart transform comparison. Geometrically, the statistic is twice the log volume of a random parallelotope spanned by standardized Gaussian coordinate vectors.
\end{abstract}


\section{Introduction}
\label{sec:introduction}

Let \(x_1,\ldots,x_n\) be independent \(p\)-dimensional Gaussian vectors,
let \(R_n\) be their population correlation matrix, and let \(\Rhat_n\) be
their Pearson sample correlation matrix. The statistic \(\log(|\Rhat_n|)\)
is proportional to the logarithm of the classical likelihood ratio for
complete independence and is the plug-in estimator of minus twice the
Gaussian total correlation
\cite{Wilks1932,Watanabe1960,RoweDay2019}. It also has an exact geometric meaning. Let
\(X=(x_1,\ldots,x_n)^\top\). After centering, the \(i\)th variable
\(X_{\cdot i}-\bar x_i\mathbf 1_n\) lies in the residual subspace
\(\mathbf 1_n^\perp\). Let \(g_i\in\mathbb R^m\), where \(m=n-1\),
denote its coordinates in an orthonormal basis, and put
\(u_i=g_i/\|g_i\|\). Then
\(\Rhat=(u_i^\top u_j)_{i,j=1}^p\), so \(|\Rhat|^{1/2}\) is the
parallelotope volume spanned by \(u_1,\ldots,u_p\). The theorem is thus a log volume limit for dependent
Gaussian directions, including near redundancy and \(m=p\), adjacent to the
random simplex and elliptical Gram literature
\cite{GusakovaHeinyThale2023,ShanLi2026}.

Finite dimensional distribution theory long predates this problem.
\cite{GuptaRathie1983} studied the determinant under a general Gaussian
population, and later work gives distributional or structural representations for
scale free scatter and correlation matrices
\cite{MathaiProvost2024,SebastianPrincy2024}. These results concern densities,
representations, or classical asymptotics, not the triangular array theorem
below.

When \(p\) diverges, beta product or determinant moment methods give null
central limit theorems in proportional regimes \cite{JiangYang2013} and at
all relative rates away from the last few gaps \cite{JiangQi2015}. The exact
square endpoint, where the \(m\times p\) residual data matrix is square, is
included in \cite{Rouault2007}. Exact digamma and trigamma moments appear in
\cite{RoweDay2019}; refinements include an Edgeworth expansion
\cite{XieSun2021} and a moderate deviation principle
\cite{BaiZhangLi2024}.

The nonidentity problem is more delicate. Passing from a sample covariance
matrix to a sample correlation matrix divides by random coordinatewise
sample variances, and these variances are mutually dependent when
\(R_n\ne I_p\). Several results solve important subregimes of this problem.
The rescaled correlation theorem of
\citet[Examples~3 and~5]{YinLiTianZheng2022} includes the
empirically centered logarithmic example
\[
\log(|R_n^{-1}\Rhat_n|)
=\log(|\Rhat_n|)-\log(|R_n|),
\]
exactly the determinant difference considered here. In its Gaussian specialization, the
variance has the same leading \(R_n\)-dependent term as ours, but the theorem
uses an interior proportional regime, independent components, analyticity, a
collection of normalized population limits involving \(R_n^{-1}\),
\(R_n^{-1/2}\), and \(R_n^{1/2}\). In the nonsingular proportional regime,
the result in \cite{ChenZhengZou2026} likewise gives the same difference by
choosing \(M_n=R_n^{-1}\) and the logarithm whenever its structural and
boundedness hypotheses hold.

The results of \citet[Theorems~2.1--2.2]{ParolyaEtAl2024}
treat an independent component model with finite fourth moments, without
assuming Gaussianity. The noncentered theorem holds under
\(p/n\to\gamma\in(0,1]\) with \(p\le n\), including the exact square
case \(p=n\). The centered theorem requires \(p<n\) and, when
\(p/n\to1\), imposes an additional approach-rate condition near the square
endpoint. Their asymptotic centering contains an
additional fourth moment correction proportional to
\[
(\mathbb E|x_{11}|^4-3)(C_{R^{1/2}}-1),
\]
which vanishes in the Gaussian case. Their assumptions require
\[
0<c\le\lambda_{\min}(R_n)
\le\lambda_{\max}(R_n)\le C<\infty
\qquad\text{uniformly in }n.
\]
Other proportional linear spectral
statistic results impose Gaussian or structural regularity
\cite{GaoEtAl2017,MestreVallet2017,YinZhengZou2023}. \cite{Li2025Unbounded} allows an unbounded population correlation matrix and
derives limiting spectral distributions and spiked eigenvalue limits, but
does not establish a central limit theorem for the sample correlation log
determinant.

An earlier Gaussian theorem covered broad relative rates of \(p\) and \(n\),
but assumed \(n>p+4\),
\(\inf_{n\ge6}\lambda_{\min}(R_n)>1/2\), and at least one of
\begin{equation}
\inf_n\frac{p_n}{n}>0,\qquad
\inf_n\frac1n\tr((R_n-I_p)^2)>0,\qquad
\sup_n\frac{p_n\|R_n-I_p\|_\infty}
{\|R_n-I_p\|_{\F}}<\infty
\label{eq:former-conditions}
\end{equation}
\cite[Theorem~1]{Jiang2019}. Those conditions support the auxiliary
determinant argument used there; they are not needed for the limiting law
proved here. In particular, a sequence may converge to \(I_p\) in operator
norm while failing every alternative in \eqref{eq:former-conditions}. Such
local sequences cannot be reduced either to a fixed alternative or to an
interior proportional theorem.

Under the null, uncentered universality extends in the interior to symmetric
regularly varying tails of index in \((3,4)\) \cite{HeinyParolya2024}, and
admits a sharp tail criterion there under
an independent entry heavy tail model
\cite{LiPanXieZhou2024}, and extends through the square case
\cite{LiLiuXieZhou2026}. These results do not transfer automatically to
empirical centering, which destroys entrywise independence outside Gaussianity.

The main distinction of this paper, formalized in
Theorem~\ref{thm:exact}, is that these restrictions are removed
simultaneously. In the Gaussian setting, this leaves the weakest possible
assumptions under which the ordinary log determinant is well-defined.
It covers every sequence \(p\to\infty\) and
\(m=n-1\ge p\), including dilute growth
(\(p/m\to0\)) and the square endpoint (\(p/m=1\)),
for an arbitrary positive definite population correlation matrix \(R_n\).

\subsection{Why correlation matrices without uniform eigenvalue bounds arise}

The absence of uniform lower and upper eigenvalue bounds on \(R_n\) is useful
in ordinary statistical models. In
a Gaussian one factor model, the equicorrelation matrix
\[
R_p=(1-\rho_p)I_p+\rho_p\1\1^\top,
\qquad
\lambda_{\min}(R_p)=1-\rho_p,
\qquad
\lambda_{\max}(R_p)=1+(p-1)\rho_p.
\]
A fixed positive \(\rho_p\) creates an order \(p\)
spike, while \(\rho_p\to1\) also drives the lower edge to zero. Dense
functional or longitudinal data provide a setting in which \(p\) is grid
resolution and \(n\) is the number of curves or subjects
\cite{HallMullerWang2006}. A concrete Gaussian model exhibits both phenomena.
For fixed \(\vartheta>0\), sampling a stationary Ornstein--Uhlenbeck trajectory
on a grid of mesh \(1/p\) gives the
Kac--Murdock--Szeg\H{o} matrix \cite{KacMurdockSzego1953}:
\[
R_p=\left(\exp(-\vartheta|i-j|/p)\right)_{i,j=1}^p.
\]
The adjacent difference Rayleigh quotient and the all ones Rayleigh
quotient, together with the maximum row sum, give
\[
\lambda_{\min}(R_p)\le1-\exp(-\vartheta/p)=O(p^{-1}),
\qquad
\lambda_{\max}(R_p)\asymp p.
\]
Such matrices are positive definite for
every finite \(p\), but lie outside theories requiring the smallest eigenvalue
to remain bounded away from zero and the largest eigenvalue to remain bounded
above.

\subsection{What is new in the proof}
\label{sec:new-proof}

The proof works directly with the random diagonal terms of the sample
covariance matrix. If
\[
Q=\sum_{\ell=1}^mG_\ell^2\sim\chi_m^2,
\qquad
h_m(G)=\log(Q)-\E(\log(Q)),
\]
then a direct Hermite calculation gives
\[
J_2h_m(G)=\frac{Q-m}{m},
\]
where \(J_2\) denotes projection onto the second Wiener chaos. After this
term is removed, the residual begins in the fourth chaos. Mehler's identity
then bounds the covariance of two residuals by the fourth power of their
underlying correlation. When the bound is summed over coordinates, the
entire nonlinear remainder is negligible after the theorem's normalization,
uniformly over \(R_n\).

Wiener chaos methods for random determinants are not new:
\cite{Notarnicola2023} developed matrix Hermite and Mehler formulas for
Gaussian determinants, and \cite{DiezTudor2023} isolated the dominant term in
the second Wiener chaos for logarithms of the independent chi square factors
in a Wishart determinant. Our literature search did not locate an earlier
central limit proof for the Gaussian Pearson correlation log determinant that
combines the exact coordinatewise second chaos projection with a Mehler
covariance bound valid for every positive definite correlation matrix \(R_n\). The
new proof strategy is this combination and its use in removing the earlier
restrictions on \(R_n\), not the projection, Mehler's identity, or Wiener
chaos for determinants by itself.

The second ingredient is an exact characteristic transform comparison. The
Wishart matrix gamma integral is classical, but the ratio between the leading
transforms under \(R_n\) and \(I_p\) reduces to
\[
\exp\left(-q\,\tr\left(\Log\left(I_p+\frac{\theta}{q}(R_n-I_p)\right)\right)\right).
\]
The second order term in the power series expansion of the log transform ratio in the complex transform parameter \(\theta\) yields the \(R_n\)-dependent contribution to the variance. Higher terms are
uniformly negligible because the theorem's normalization controls the
Frobenius norm of \(R_n-I_p\). No lower or upper eigenvalue bound on \(R_n\),
and no effective rank condition, enters the estimate.

\subsection{Contributions}
\label{sec:contributions}

\begin{enumerate}
\item We prove a Gaussian log determinant central limit theorem for every
positive definite \(R_n\), every sequence \(p\to\infty\), and every
\(m=n-1\ge p\).
\item We use the exact finite sample mean and the known polygamma null
variance, thereby including the square endpoint \(m=p\), where the elementary
normalization used away from the boundary is undefined.
\item We derive an explicit finite sample variance series valid for every
positive definite \(R_n\)
from the classical Kibble bivariate gamma transform and prove a relative
approximation bound whose constants do not depend on \(R_n\).
\item We give a self-contained proof of the second chaos projection and all
error estimates uniform over positive definite \(R_n\). Established Wishart, beta product, and
null central limit theorem inputs are identified by citation.
\end{enumerate}

Relative to the cited log determinant theorems, the contribution is one
Gaussian result covering dilute asymptotics, arbitrary positive definite
\(R_n\), and the hard edge with a common exact centering. Throughout the
paper, ``uniform over \(R_n\)'' refers to deterministic error bounds whose
constants do not depend on the correlation matrix \(R_n\). We do not claim a uniform
Berry--Esseen or Kolmogorov distance bound.

\section{Related results and the hard edge}
\label{sec:history}

\subsection{Determinant limit theory}

Under the null, Gram--Schmidt orthogonalization factors the determinant into
independent beta variables \cite{Rouault2007}. The same factorization yields
the exact polygamma normalization \cite{RoweDay2019}. For sample covariance matrices, log
determinants have been treated when the aspect ratio \(p/m\) approaches one and at
the square endpoint \cite{WangHanPan2018}. These
results concern a covariance determinant; they do not contain the random
diagonal normalization of a Pearson correlation matrix.

Two adjacent geometric papers study Gram determinants because the logarithmic
volume of a random simplex is a log determinant. The first gives interior
proportional limits for elliptical observations under spectral and
near isotropy conditions \cite{GusakovaHeinyThale2023}. A recent note weakens
the near isotropy requirement but retains an interior proportional regime and
lower and upper eigenvalue bounds \cite{ShanLi2026}. These results motivate
the removal of restrictions on \(R_n\), but their statistic is not the Pearson sample
correlation determinant.

\subsection{The hard edge and the square endpoint}
\label{sec:hard-edge}

The phrase \emph{hard edge} comes from random matrix theory.  For a normalized
\(p\times p\) Wishart matrix with \(p/m\to\gamma\in(0,1)\), the
Marchenko--Pastur limiting spectrum is supported on
\[
\left[(1-\sqrt{\gamma})^2,(1+\sqrt{\gamma})^2\right].
\]
This is the Marchenko--Pastur law \cite{MarchenkoPastur1967}.
When \(\gamma<1\), the lower endpoint is positive.  As
\(\gamma\uparrow1\), it reaches zero.  Eigenvalues of a positive
semidefinite matrix cannot cross zero, so zero is a fixed, or ``hard,''
boundary. Its microscopic limits are of Bessel type. For the Laguerre and
Jacobi unitary ensembles, \cite{TracyWidom1994Bessel} studied the
Bessel-kernel Fredholm determinant.

Three related regimes should be distinguished:
\[
\begin{array}{ll}
\text{square endpoint:} & m=p,\\
\text{microscopic hard edge:} & m-p=O(1),\\
\text{near hard edge:} & p/m\to1.
\end{array}
\]
The logarithm is singular at zero, so small eigenvalues have an unusually
large effect on \(\log(|\Rhat_n|)\).  The exact null variance \(v_{m,p}\) makes this
transition explicit. Lemma~\ref{lem:app-normalization}, the power series for
\(-\log(1-y)\), and the trigamma expansion in
\cite[Eq.~(5.15.8), p.~144]{OlverEtAl2010} give
\[
\begin{aligned}
v_{m,p}&\sim \frac{p^2}{m^2}
&&\text{if }p/m\to0,\\
v_{m,p}&\longrightarrow-2\{y+\log(1-y)\}
&&\text{if }p/m\to y\in(0,1),\\
v_{m,p}&=2\log(p)+O(1)
&&\text{if }m-p=O(1).
\end{aligned}
\]
The last line includes \(m=p\).
Thus an argument proved only for \(p/m\to\gamma<1\) cannot simply be
evaluated at \(\gamma=1\).
The null hard edge central limit theorem is prior work
\cite{Rouault2007,RoweDay2019,LiLiuXieZhou2026}. The extension here is to
empirically centered Gaussian Pearson matrices under arbitrary positive
definite \(R_n\).

\subsection{Comparison of representative results}

\begin{table}[htbp]
\centering
\scriptsize
\begin{tabularx}{\textwidth}{@{}L{0.18\textwidth}L{0.24\textwidth}
 L{0.25\textwidth}Y@{}}
\toprule
Source & Statistic/regime & Matrix and distributional assumptions
& Relation to this paper\\
\midrule
\cite{Rouault2007}
& Uniform Gram determinant, including the square endpoint \(m=p\)
& Independent vectors uniform on the real unit sphere; equivalent to the
Gaussian null model after centering and column normalization
& The beta product and square endpoint null CLT are established inputs.\\

\cite{RoweDay2019}
& Gaussian total correlation and exact null moments
& Independent standard Gaussian coordinates with known zero mean and an
uncentered sample covariance matrix
& With their dimension and sample size identified as \(p\) and \(m\), and
using \(\log(|\widehat R|)=-2\widehat T\), gives the exact null
digamma/trigamma normalization.\\

\cite{WangHanPan2018}
& Sample covariance log determinant near and at the square endpoint
& \(X_k=\Sigma^{1/2}Y_k\) with deterministic positive definite \(\Sigma\);
the entries of \(Y_k\) are independent, mean zero, variance one, and satisfy
a uniform fourth moment bound
& Gives near square and square sample covariance CLTs without Pearson
normalization.\\

\cite{YinLiTianZheng2022}
& Linear spectral statistics of \(R_n^{-1}\widehat R_n\); the logarithm gives
the same determinant difference when \(p/n\to\rho\in(0,1)\)
& Independent components, analytic test functions, a log-weighted fourth
moment condition, and existence of the five normalized population limits
in Assumption~6
& Contains the same asymptotic Gaussian \(R_n\)-dependent variance term in
its regime.\\

\cite{HeinyParolya2024}
& Uncentered null sample correlation log determinant,
\(p/n\to\gamma\in(0,1)\)
& I.i.d.\ symmetric entries with regularly varying tails of index
\(\alpha\in(3,4)\)
& Establishes a heavy-tailed null universality result in the proportional
interior.\\

\cite{ParolyaEtAl2024}
& Noncentered log determinant for \(p/n\to\gamma\in(0,1]\); sample-mean-centered
version for \(p<n\), with \(1-p/n=O(n^{-1/12})\) when \(\gamma=1\)
& I.i.d.\ mean-zero, variance-one components with finite fourth moment and,
in the published theorem,
\(c\le\lambda_{\min}(R_n)\le\lambda_{\max}(R_n)\le C\)
& Covers non-Gaussian alternatives in proportional regimes under two sided
spectral bounds.\\

\cite{ChenZhengZou2026}
& Linear spectral statistics of \(\widehat R_nM_n\); specializing to
\(M_n=R_n^{-1}\) and the logarithm gives the determinant difference when
\(p/n\to y\in(0,1)\)
& Proportional growth, analytic test functions, a linear independent-component
model, convergence of the empirical spectral distribution of \(R_nM_n\),
and uniform bounds on \(\|R_n\|\), \(\|M_n\|\), and \(\|M_n^{-1}\|\)
& Does not cover \(p/n\to0\), the square endpoint for the logarithm, or
unbounded \(\|R_n\|\) or \(\|R_n^{-1}\|\).\\

\cite{LiLiuXieZhou2026}
& Uncentered null sample correlation log determinant,
\(p/n\to\gamma\in(0,1]\), including \(p=n\)
& I.i.d.\ mean-zero, variance-one entries with regularly varying tails
satisfying
\(\lim_{x\to\infty}x^3\Pr(|\xi|>x)=0\); symmetry is not required
& Extends non-Gaussian null universality to the square endpoint, but does
not cover \(p/n\to0\) or nonidentity population correlation.\\

\cite{ShanLi2026}
& Spherical-direction Gram determinant underlying elliptical simplex
volumes, \(p/n\to\gamma\in(0,1)\)
& Independent uniform directions, uniform spectral bounds on
\(A^{\mathsf T}A\), and both eigenvalue-dispersion conditions in
Assumption~2.2
& Gives an adjacent Gram determinant CLT without Pearson normalization.\\

This paper
& Centered Pearson sample correlations under Gaussian alternatives; every
\(p\to\infty\) with \(m\ge p\)
& Every positive definite population correlation matrix \(R_n\)
& No uniform lower or upper eigenvalue bound on \(R_n\), and no condition on
\(\tr((R_n-I_p)^2)\), sparsity, or individual entries.\\
\bottomrule
\end{tabularx}
\caption{Representative determinant and sample correlation results. The rows
concerning sample covariance and elliptical Gram determinants address adjacent
determinant problems rather than Pearson sample correlation matrices.}
\label{tab:history}
\end{table}

\section{Model, notation, and main results}
\label{sec:main-results}

Let \(x_1,\ldots,x_n\) be independent \(N_p(\mu,\Sigma)\) observations.
Write \(D_\Sigma=\diag(\Sigma_{11},\ldots,\Sigma_{pp})\) and
\(R=D_\Sigma^{-1/2}\Sigma D_\Sigma^{-1/2}\), the population correlation
matrix. The centered scatter matrix and the Pearson sample correlation matrix
are
\[
S=\sum_{k=1}^n(x_k-\bar x)(x_k-\bar x)^\top,
\qquad
\Rhat=\diag(S)^{-1/2}S\diag(S)^{-1/2}.
\]

Put \(m=n-1\),
\(A=R-I_p\), and \(a=\tr(A^2)=\|A\|_{\mathrm F}^2\).
The diagonal of \(A\) is zero, so \(\tr(A)=0\).  Both \(p=p_n\) and
\(R=R_n\) may vary with \(n\).

Let \(\psi(x)=\Gamma'(x)/\Gamma(x)\) and
\(\psi_1(x)=\psi'(x)\) denote the digamma and trigamma functions.  For
integers \(m\ge p\), define
\begin{align}
b_{m,p}
&=\sum_{j=2}^p
\left\{\psi\left(\frac{m-j+1}{2}\right)-\psi\left(\frac m2\right)\right\},
\label{eq:bmp}\\
v_{m,p}
&=\sum_{j=2}^p
\left\{\psi_1\left(\frac{m-j+1}{2}\right)
-\psi_1\left(\frac m2\right)\right\}.
\label{eq:vmp}
\end{align}
Under \(R=I_p\), these are the exact finite sample mean and variance of
\(\log(|\Rhat|)\).  They follow from the classical beta-product or determinant
Mellin transform and appear explicitly, under equivalent sample size
conventions, in \cite{RoweDay2019,XieSun2021}; related earlier moment
formulas and their correction are discussed by
\cite{Guerrero1994,RoweDay2019}.  Because centering the data removes one
observation direction, the degrees of freedom in this paper are \(m=n-1\).

\begin{theorem}[Gaussian log determinant CLT]
\label{thm:exact}
Assume that \(p=p_n\to\infty\), \(m\ge p\), and
\(\{R_n\}_{n=1}^\infty\) is any sequence of positive definite \(p\times p\)
correlation matrices.  Then
\begin{equation}
\frac{\log(|\Rhat_n|)-\{\log(|R_n|)+b_{m,p}\}}
{\left\{v_{m,p}+\dfrac2m\tr\!\left((R_n-I_p)^2\right)\right\}^{1/2}}
\ \Rightarrow\ N(0,1).
\label{eq:main-clt}
\end{equation}

\end{theorem}
No uniform lower or upper eigenvalue bound on \(R_n\) and no sparsity, density,
effective rank, or entrywise condition is required.

The numerator in \eqref{eq:main-clt} uses the exact finite sample mean.
The denominator is not quite the exact alternative variance, but the next
proposition shows that it is uniformly relatively accurate.  The detailed
calculation is organized as Lemmas~\ref{lem:app-crosscov},
\ref{lem:app-kibble}, and \ref{lem:app-exact-variance} in Appendix~\ref{app:exact-variance}.

\begin{proposition}[Exact finite sample variance]
\label{prop:exact-cumulants}
Let \(\alpha=m/2\), and for \(|r|\le1\) define
\begin{equation}
c_m(r)
=\sum_{k=1}^{\infty}
\frac{(k-1)!}{k(\alpha)_k}r^{2k},
\label{eq:cm-def}
\end{equation}
where \((\alpha)_k=\alpha(\alpha+1)\cdots(\alpha+k-1)\).  If
\(R=(r_{ij})\), then
\begin{equation}
\tau_{m,p}^2(R)
:=\Var_R(\log(|\Rhat|))
=v_{m,p}+\sum_{i\ne j}c_m(r_{ij}).
\label{eq:exact-var}
\end{equation}
Moreover,
\begin{equation}
0\le
\tau_{m,p}^2(R)
-\left\{v_{m,p}+\frac2m\tr\!\left((R-I_p)^2\right)\right\}
\le\frac4{m^2}\tr\!\left((R-I_p)^2\right).
\label{eq:variance-comparison}
\end{equation}
Consequently, the exact scale \(\tau_{m,p}(R)\) may replace the denominator
in Theorem~\ref{thm:exact}.
\end{proposition}

\begin{remark}
The Wishart and Kibble bivariate-gamma inputs behind
\eqref{eq:exact-var} are classical
\cite{Muirhead1982,Kibble1941,NadarajahKotz2006}. The classical Wishart
transform also contains the same finite sample information.
Appendix~\ref{app:exact-variance}
records the explicit mixed differentiation, covariance cancellation, and
uniform remainder. We claim the explicit synthesis and bound, not a new Wishart
distributional identity. Earlier exact and structural distribution results
are discussed in Section~\ref{sec:history}; Proposition~\ref{prop:exact-cumulants}
claims the displayed variance series and comparison bound, not priority for
finite sample distribution theory.
\end{remark}

The Gaussian conclusion has one exact matrix spherical extension. This does
not produce independent non-Gaussian observations, but it identifies a larger
class for which the entire finite sample law of the Pearson matrix is
unchanged.

\begin{corollary}[Matrix spherical extension]
\label{cor:matrix-spherical}
Let \(U=U_{p,n}\) be uniform on the Frobenius unit sphere in
\(\mathbb R^{p\times n}\), let \(T=T_{p,n}>0\) be independent of \(U\),
and define
\[
Y=\mu\1_n^\top+T\Sigma^{1/2}U.
\]
Let the columns of \(Y\) be the observations, where \(\Sigma\) is positive
definite. Form the Pearson
matrix by empirical centering, and let \(R\) be the correlation matrix of
\(\Sigma\). Then this Pearson matrix has exactly the same distribution as in
the independent Gaussian model with covariance \(\Sigma\). Consequently,
Theorem~\ref{thm:exact} and Proposition~\ref{prop:exact-cumulants} hold without
change whenever \(p\to\infty\) and \(n-1\ge p\).
\end{corollary}

\begin{proof}
Put \(H=I_n-n^{-1}\1_n\1_n^\top\). On an enlarged probability space, take
\(\rho>0\), independent of \((T,U)\), with \(\rho^2\sim\chi^2_{pn}\). The polar
decomposition of a standard Gaussian vector, after vectorization, gives
\(G:=\rho U\) with independent \(N(0,1)\) entries; \cite{Dawid1977}
develops the surrounding spherical matrix framework. Since
\[
YH=\frac{T}{\rho}\Sigma^{1/2}GH,
\]
its scatter matrix is the Gaussian scatter
matrix multiplied by the single positive scalar \((T/\rho)^2\). Pearson
normalization cancels that scalar exactly.
\end{proof}

For \(m>p\), an elementary normalization used in earlier high dimensional
likelihood ratio theory \cite{JiangQi2015} is
\begin{align}
\mu_{0,m,p}
&=\left(p-m+\frac12\right)\log\left(1-\frac pm\right)
-\frac{m-1}{m}p,
\label{eq:elementary-null-mean}\\
\sigma_{0,m,p}^2
&=-2\left\{\frac pm+\log\left(1-\frac pm\right)\right\}.
\label{eq:elementary-null-var}
\end{align}
The corresponding general-\(R\) center and variance are
\[
\mu_{m,p}(R)=\mu_{0,m,p}+\log(|R|),
\qquad
\sigma_{m,p}^2(R)
=\sigma_{0,m,p}^2+\frac{2}{m}\tr((R-I_p)^2).
\]

\begin{corollary}[Elementary normalization]
\label{cor:elementary}
Assume \(p\to\infty\), \(m>p\), and \(R_n\) is any
positive definite correlation matrix.  Then
\begin{equation}
\frac{\log(|\Rhat_n|)-\mu_{m,p}(R_n)}
{\sigma_{m,p}(R_n)}
\Rightarrow N(0,1).
\label{eq:elementary-corollary}
\end{equation}
\end{corollary}

The exact polygamma normalization is needed at \(m=p\), because
\(\log(1-p/m)\) in
\eqref{eq:elementary-null-mean}--\eqref{eq:elementary-null-var}
is not finite. For every \(m>p\), Lemma~\ref{lem:normalization-equivalence} and
Appendix~\ref{app:normalization} prove that the two normalizations are
asymptotically equivalent, including the fixed gaps \(m-p=1,2,3\).

\section{A self-contained reminder on Gaussian Wiener chaos}
\label{sec:chaos-primer}

The words ``Wiener chaos'' can make the proof sound more abstract than it is.
In the finite dimensional Gaussian setting used here, a Wiener chaos
expansion is simply an orthogonal expansion in multivariate Hermite
polynomials.  The construction goes back to \cite{Wiener1938}; the facts
used below are standard and can be found in
\cite[Chapters~2 and~10]{Janson1997}.  We recall them so that the
calculation in Section~\ref{sec:proof} can be checked without prior
Malliavin calculus training.

For \(k\ge0\), the probabilists' Hermite polynomial is
\[
H_k(x)=(-1)^ke^{x^2/2}\frac{\dd^k}{\dd x^k}e^{-x^2/2},
\qquad
H_0(x)=1,\quad H_1(x)=x,\quad H_2(x)=x^2-1.
\]
If \(G=(G_1,\ldots,G_m)^\top\sim N_m(0,I_m)\) and
\(\alpha=(\alpha_1,\ldots,\alpha_m)\in\mathbb N_0^m\), write
\[
|\alpha|=\sum_{\ell=1}^m\alpha_\ell,
\qquad
\alpha!=\prod_{\ell=1}^m\alpha_\ell!,
\qquad
H_\alpha(x)=\prod_{\ell=1}^mH_{\alpha_\ell}(x_\ell).
\]
Independence of the coordinates and the one dimensional Hermite
orthogonality relation give
\begin{equation}
\E(H_\alpha(G)H_\beta(G))
=\1_{(\alpha=\beta)}\alpha!.
\label{eq:hermite-orthogonality}
\end{equation}

Every square integrable function \(f(G)\) has the \(L^2\)-convergent
orthogonal expansion
\begin{equation}
f(G)=\sum_{\alpha\in\mathbb N_0^m}
\widehat f(\alpha)H_\alpha(G),
\qquad
\widehat f(\alpha)
=\frac{\E(f(G)H_\alpha(G))}{\alpha!}.
\label{eq:hermite-expansion}
\end{equation}
The projection
\[
J_kf=\sum_{|\alpha|=k}\widehat f(\alpha)H_\alpha
\]
is called the \(k\)th Wiener chaos component. Thus \(J_0f=\E(f(G))\);
\(J_1f\) is the best linear Gaussian approximation; and \(J_2f\) is the
orthogonal quadratic approximation. When \(\E(f(G))=0\), orthogonality
implies
\[
\Var(f(G))=\sum_{k=1}^\infty\|J_kf\|_2^2.
\]

The second standard input is the one dimensional correlated Hermite
identity of
\cite[Proposition~2.2.1, p.~26]{NourdinPeccati2012}.  It can also be
obtained from the Ornstein--Uhlenbeck spectral definition and Mehler
representation in
\cite[Definition~2.8.1 and Theorem~2.8.2,
pp.~45--47]{NourdinPeccati2012}, together with Hermite orthogonality.
Those cited pages do not state the multivariate formula below verbatim.
For completeness, Lemma~\ref{lem:app-mehler-hermite} in
Appendix~\ref{app:mehler-hermite} derives the identity, including negative
values of \(\rho\) and the endpoints \(\rho=\pm1\).

Suppose that \(G'\) is also \(N_m(0,I_m)\) and
\[
\Cov(G,G')=\rho I_m,
\qquad |\rho|\le1.
\]
Then Lemma~\ref{lem:app-mehler-hermite} gives
\begin{equation}
\E(H_\alpha(G)H_\beta(G'))
=\1_{(\alpha=\beta)}\alpha!\rho^{|\alpha|}.
\label{eq:mehler-hermite}
\end{equation}
Substituting the two expansions \eqref{eq:hermite-expansion} into a
covariance and using \eqref{eq:mehler-hermite} gives
\begin{equation}
\Cov(f(G),f(G'))
=\sum_{k=1}^\infty\rho^k\|J_kf\|_2^2
\label{eq:mehler-covariance}
\end{equation}
for every centered \(f\in L^2(N_m(0,I_m))\).  Notice the key consequence:
if the first nonzero chaos of \(f\) has degree four, then every term in
\eqref{eq:mehler-covariance} contains at least \(\rho^4\).  This fourth
power is what makes the remainder bound uniform over \(R\) possible.

\section{Proof of Theorem~\ref{thm:exact}}
\label{sec:proof}

\subsection{Wishart representation and exact centering}
\label{sec:proof-wishart}

The first lemma is established Wishart theory.  It is stated separately to
identify exactly which part of the proof is quoted from the literature.

\begin{lemma}[Classical Wishart reduction]
\label{lem:wishart-mean}
Let \(m\ge p\). There is a \(p\times m\) matrix \(Z\) with independent
\(N(0,1)\) entries such that
\[
W_0=ZZ^\top,
\qquad
W_R=R^{1/2}W_0R^{1/2}\sim W_p(m,R).
\]
If \(Q_i=(W_R)_{ii}\), then
\begin{equation}
\log(|\Rhat|)
=\log(|R|)+\log(|W_0|)-\sum_{i=1}^p\log(Q_i),
\label{eq:wishart-decomp}
\end{equation}
\(Q_i\sim\chi_m^2\), and
\begin{equation}
\E(\log(|\Rhat|))=\log(|R|)+b_{m,p}.
\label{eq:exact-mean}
\end{equation}
\end{lemma}

\begin{proof}
The centered Gaussian scatter matrix has the Wishart law \(W_p(m,\Sigma)\)
by \cite[Theorem~3.1.2]{Muirhead1982}. After diagonal rescaling, Wishart
equivariance shows that its correlation matrix has the same distribution as
the correlation matrix formed from \(W_R\sim W_p(m,R)\).

We include the remaining determinant algebra. Since
\[
\Rhat=\diag(W_R)^{-1/2}W_R\diag(W_R)^{-1/2},
\]
multiplicativity of the determinant gives
\[
|\Rhat|
=\frac{|W_R|}{\prod_{i=1}^p(W_R)_{ii}}
=\frac{|R|\,|W_0|}{\prod_{i=1}^pQ_i},
\]
which proves \eqref{eq:wishart-decomp}. The equality
\[
|W_R|=|R^{1/2}W_0R^{1/2}|=|R^{1/2}|^2|W_0|=|R||W_0|
\]
follows directly.

For each \(i\), \(Q_i\) is a sum of squares of \(m\) independent
standard normal variables, because the \(i\)th marginal variance of \(R\)
is one. Hence \(Q_i\sim\chi_m^2\). Bartlett's decomposition gives
\[
|W_0|\stackrel d=\prod_{j=1}^p\chi^2_{m-j+1}
\]
with independent factors
\cite[Theorem~3.2.14]{Muirhead1982}. Direct integration of the
\(\chi_\nu^2\) gamma density gives
\[
\E((\chi_\nu^2)^t)
=2^t\frac{\Gamma(\nu/2+t)}{\Gamma(\nu/2)},
\qquad
\E(\log(\chi_\nu^2))=\log(2)+\psi(\nu/2).
\]
Consequently,
\[
\E(\log(|W_0|))
=p\log(2)+\sum_{j=1}^p\psi\left(\frac{m-j+1}{2}\right),
\qquad
\sum_{i=1}^p\E(\log(Q_i))
=p\left\{\log(2)+\psi\left(\frac m2\right)\right\}.
\]
Subtracting these expressions in \eqref{eq:wishart-decomp}, and observing
that the \(j=1\) digamma term cancels one copy of \(\psi(m/2)\), gives
\eqref{eq:exact-mean} and the definition \eqref{eq:bmp}.
\end{proof}

\subsection{Direct computation of the second chaos projection}
\label{sec:proof-chaos}

The next lemma is the central use of chaos in the present proof.  The general
Hermite expansion and Mehler identity are quoted in
Section~\ref{sec:chaos-primer}; every coefficient needed for the particular
function \(\log(\|G\|^2)\) is computed below.

\begin{lemma}[Exact second chaos projection and residual covariance]
\label{lem:second-chaos}
Let \(G\sim N_m(0,I_m)\), \(Q=\|G\|^2\), and
\[
h_m(G)=\log(Q)-\E(\log(Q)),
\qquad
u_m(G)=\frac{Q-m}{m},
\qquad
e_m(G)=h_m(G)-u_m(G).
\]
Then \(u_m=J_2h_m\).  Moreover,
\begin{align}
&\Var(e_m)
=\psi_1(m/2)-\frac2m\le\frac4{m^2},
\label{eq:e-var}\\
&0\le\Cov(e_m(G),e_m(G'))
\le\rho^4\Var(e_m)
\le\frac{4\rho^4}{m^2}
\label{eq:mehler}
\end{align}
whenever \((G,G')\) is jointly Gaussian with
\(G,G'\sim N_m(0,I_m)\) and
\(\Cov(G,G')=\rho I_m\), where \(|\rho|\le1\).
\end{lemma}

\begin{proof}
Put \(\alpha=m/2\).  Since \(Q\sim\chi_m^2\), its Mellin transform is
\begin{equation}
M(t)=\E(Q^t)
=2^t\frac{\Gamma(\alpha+t)}{\Gamma(\alpha)},
\qquad t>-\alpha.
\label{eq:chi-mellin}
\end{equation}
Differentiating gives
\begin{equation}
M'(t)=M(t)\{\log(2)+\psi(\alpha+t)\}.
\label{eq:chi-mellin-derivative}
\end{equation}
At \(t=0\), \(M(0)=1\), so
\[
\E(\log(Q))=M'(0)=\log(2)+\psi(\alpha).
\]
At \(t=1\), \(M(1)=\E(Q)=m\), and hence
\[
\E(Q\log(Q))=M'(1)=m\{\log(2)+\psi(\alpha+1)\}.
\]
The digamma recurrence
\(\psi(x+1)-\psi(x)=1/x\) now gives
\[
\Cov(\log(Q),Q)
=\E(Q\log(Q))-\E(Q)\E(\log(Q))
=m\{\psi(\alpha+1)-\psi(\alpha)\}
=\frac m\alpha=2.
\]
Likewise,
\[
\E(Q^2)=M(2)=4\frac{\Gamma(\alpha+2)}{\Gamma(\alpha)}
=4\alpha(\alpha+1)=m(m+2),
\qquad
\Var(Q)=2m.
\]

We next identify \(J_2h_m\) coefficient by coefficient, without invoking
rotational invariance. The second chaos is spanned by
\[
H_2(G_\ell)=G_\ell^2-1,\quad 1\le\ell\le m,
\qquad
H_1(G_\ell)H_1(G_k)=G_\ell G_k,\quad \ell<k.
\]
For a fixed \(\ell\), exchangeability of the coordinates gives
\[
\E(G_\ell^2Q^t)
=\frac1m\E\left(\left(\sum_{r=1}^mG_r^2\right)Q^t\right)
=\frac1mM(t+1).
\]
Differentiating at \(t=0\) and using
\eqref{eq:chi-mellin-derivative} gives
\[
\E(G_\ell^2\log(Q))
=\frac1mM'(1)=\log(2)+\psi(\alpha+1).
\]
Since \(\E(G_\ell^2)=1\),
\[
\E(h_m(G)H_2(G_\ell))
=\Cov(\log(Q),G_\ell^2)
=\psi(\alpha+1)-\psi(\alpha)
=\frac1\alpha=\frac2m.
\]
Also \(\E(H_2(G_\ell)^2)=2\). By the projection formula in
\eqref{eq:hermite-expansion}, the coefficient of \(H_2(G_\ell)\) is
therefore \(1/m\).

For \(\ell\ne k\),
\[
\E(h_m(G)G_\ell G_k)=0.
\]
Indeed, replace
\(G_\ell\) by \(-G_\ell\) in the defining Gaussian
integral.  The quantity \(Q\), the function \(h_m(G)\), and the Gaussian
density remain unchanged, while \(G_\ell G_k\) changes sign. Thus every
mixed second-chaos coefficient is zero. Combining these coefficients gives
\[
J_2h_m(G)
=\frac1m\sum_{\ell=1}^mH_2(G_\ell)
=\frac1m\sum_{\ell=1}^m(G_\ell^2-1)
=\frac{Q-m}{m}=u_m(G).
\]

The function \(h_m\) is centered, so \(J_0h_m=0\).  More strongly, it is
even in each coordinate separately.  Therefore a Hermite coefficient
vanishes whenever any component of its multi-index is odd.  In particular,
the first and third chaoses vanish.  After \(J_2h_m\) is subtracted, the
expansion of \(e_m\) begins in total degree four.

Orthogonality of different chaoses yields
\[
\Var(e_m)=\Var(h_m)-\Var(u_m).
\]
Differentiating the cumulant-generating function \(\log M(t)\) twice
at zero gives the second cumulant,
\(\Var(\log Q)=\psi_1(\alpha)\), while
\(\Var(Q)=2m\) gives \(\Var(u_m)=2/m\). Hence the equality in
\eqref{eq:e-var} follows.
The trigamma series
\[
\psi_1(x)=\sum_{r=0}^\infty(x+r)^{-2}
\]
is given by
\cite[Eq.~(5.15.1), p.~144]{OlverEtAl2010}.  Because
\(t\mapsto(x+t)^{-2}\) is decreasing,
\[
\sum_{r=1}^\infty(x+r)^{-2}
\le\int_0^\infty(x+t)^{-2}\dd t
=\frac1x.
\]
Thus \(\psi_1(x)\le x^{-2}+x^{-1}\); setting \(x=m/2\) proves
\(\Var(e_m)\le4/m^2\).

Finally, write
\[
e_m=\sum_{q=2}^\infty J_{2q}e_m.
\]
Mehler's identity \eqref{eq:mehler-covariance} gives
\[
\Cov(e_m(G),e_m(G'))
=\sum_{q=2}^\infty\rho^{2q}\|J_{2q}e_m\|_2^2.
\]
Every coefficient is nonnegative, and
\(|\rho|^{2q}\le|\rho|^4\) for \(q\ge2\).  Therefore
\[
0\le\Cov(e_m(G),e_m(G'))
\le\rho^4\sum_{q=2}^\infty\|J_{2q}e_m\|_2^2
=\rho^4\Var(e_m),
\]
which proves \eqref{eq:mehler}.
\end{proof}

\subsection{Construction of the coordinate vectors and the uniform remainder}
\label{sec:proof-remainder}

Let \(z_1,\ldots,z_m\) be the columns of \(Z\).  Thus the \(z_k\) are
independent \(N_p(0,I_p)\) vectors.  For \(1\le i\le p\), let \(e_i\)
denote the \(i\)th coordinate vector and define
\begin{equation}
v_i=R^{1/2}e_i,\qquad
g_i=Z^\top v_i
=\bigl(v_i^\top z_1,\ldots,v_i^\top z_m\bigr)^\top.
\label{eq:gi-definition}
\end{equation}
Because \(R\) is a correlation matrix,
\[
\|v_i\|^2=e_i^\top Re_i=r_{ii}=1,
\]
so \(g_i\sim N_m(0,I_m)\).
For observation indices \(k,\ell\),
\[
\Cov(g_{i,k},g_{j,\ell})
=\1_{(k=\ell)}v_i^\top v_j
=\1_{(k=\ell)}e_i^\top Re_j
=\1_{(k=\ell)}r_{ij}.
\]
Consequently,
\begin{equation}
\Cov(g_i,g_j)=r_{ij}I_m.
\label{eq:gi-covariance}
\end{equation}
Direct matrix multiplication also gives
\begin{align}
Q_i=(W_R)_{ii}
=e_i^\top R^{1/2}ZZ^\top R^{1/2}e_i =v_i^\top ZZ^\top v_i
=\|Z^\top v_i\|^2=\|g_i\|^2.
\label{eq:Qi-gi}
\end{align}

\begin{lemma}[Nonlinear-remainder bound uniform over correlation matrices]
\label{lem:remainder}
Set
\[
E_R=\sum_{i=1}^pe_m(g_i),
\qquad
s_R^2=v_{m,p}+\frac{2a}{m}.
\]
Then
\begin{equation}
\E\left(\left(\frac{E_R}{s_R}\right)^2\right)
\le\frac4{p-1}+\frac2m\longrightarrow0.
\label{eq:uniform-remainder}
\end{equation}
In particular, \(E_R/s_R\to0\) in \(L^2\), uniformly over all positive definite correlation matrices \(R\).
\end{lemma}

\begin{proof}
Each \(e_m(g_i)\) is centered, so
\[
\Var(E_R)=\sum_{i=1}^p\sum_{j=1}^p
\Cov(e_m(g_i),e_m(g_j)).
\]
Lemma~\ref{lem:second-chaos} with
\(\rho=r_{ij}\) gives
\[
0\le\Cov(e_m(g_i),e_m(g_j))\le\frac{4r_{ij}^4}{m^2}.
\]
Since every correlation satisfies \(|r_{ij}|\le1\),
\[
\Var(E_R)\le\frac4{m^2}\sum_{i,j}r_{ij}^4
\le\frac4{m^2}\sum_{i,j}r_{ij}^2.
\]
Because \(R\) is symmetric,
\[
\sum_{i,j}r_{ij}^2=\tr(R^2).
\]
Writing \(R=I_p+A\), using \(\tr(A)=0\), and recalling
\(\tr(A^2)=a\), we obtain
\[
\tr(R^2)=\tr((I_p+A)^2)=p+2\tr(A)+\tr(A^2)=p+a.
\]
Thus
\begin{equation}
\Var(E_R)\le\frac4{m^2}(p+a).
\label{eq:remainder-var}
\end{equation}

We now establish the lower bound for the null variance that is used
throughout the paper:
\begin{equation}
v_{m,p}\ge\frac{p(p-1)}{m^2}.
\label{eq:v-lower}
\end{equation}
The series in
\cite[Eq.~(5.15.1), p.~144]{OlverEtAl2010} and its derivative converge
locally uniformly for \(x>0\).  Termwise differentiation therefore gives
\[
-\psi_2(x)=2\sum_{k=0}^\infty(x+k)^{-3}.
\]
Because the summand is decreasing,
\[
2\sum_{k=0}^\infty(x+k)^{-3}
\ge2\int_0^\infty(x+t)^{-3}\dd t=x^{-2}.
\]
For \(1\le r\le p-1\), the fundamental theorem of calculus therefore gives
\begin{align*}
\psi_1\left(\frac{m-r}{2}\right)-\psi_1\left(\frac m2\right)
&=\int_{(m-r)/2}^{m/2}\{-\psi_2(x)\}\dd x \ge\int_{(m-r)/2}^{m/2}x^{-2}\dd x\\
&=\frac2{m-r}-\frac2m
=\frac{2r}{m(m-r)}
\ge\frac{2r}{m^2}.
\end{align*}
Summing over \(r=1,\ldots,p-1\) proves \eqref{eq:v-lower}.

Use \eqref{eq:remainder-var} and split its numerator into \(p\) and \(a\).
If \(a>0\),
\begin{align*}
\frac{\Var(E_R)}{s_R^2}
\le
\frac{4p/m^2}{v_{m,p}}
+\frac{4a/m^2}{2a/m} \le\frac4{p-1}+\frac2m.
\end{align*}
When \(a=0\), the second fraction is absent and the same bound holds.
Since \(p\to\infty\) and \(m\ge p\), the right hand side tends to zero.
The bound contains no eigenvalue or other \(R\)-dependent quantity.
\end{proof}

\subsection{Reduction to a linear correlation-matrix term}
\label{sec:proof-reduction}

We now derive the leading statistic without suppressing any algebra.  From
\eqref{eq:wishart-decomp} and Lemma~\ref{lem:wishart-mean},
\begin{align*}
\log(|\Rhat|)-\E(\log(|\Rhat|))
&=\log(|W_0|)-\E(\log(|W_0|))
  -\sum_{i=1}^p\{\log(Q_i)-\E(\log(Q_i))\}.
\end{align*}
By Lemma~\ref{lem:second-chaos} and \eqref{eq:Qi-gi},
\[
\log(Q_i)-\E(\log(Q_i))=\frac{Q_i-m}{m}+e_m(g_i).
\]
Substitution gives
\begin{align}
\log(|\Rhat|)-\E(\log(|\Rhat|))
&=\log(|W_0|)-\E(\log(|W_0|))
  -\frac{\sum_{i=1}^pQ_i-mp}{m}
  -\sum_{i=1}^pe_m(g_i).
\label{eq:reduction-expanded}
\end{align}
The sum of the diagonal elements is the trace, so
\[
\sum_{i=1}^pQ_i
=\tr(W_R)=\tr(R^{1/2}W_0R^{1/2})=\tr(RW_0).
\]
The last equality is the cyclic identity
\(\tr(BC)=\tr(CB)\).  Define
\begin{equation}
M_R=\log(|W_0|)-\E(\log(|W_0|))
-\frac{\tr(RW_0)-mp}{m}.
\label{eq:MR}
\end{equation}
Equation \eqref{eq:reduction-expanded} is now the exact decomposition
\begin{equation}
\log(|\Rhat|)-\E(\log(|\Rhat|))=M_R-E_R.
\label{eq:M-minus-E}
\end{equation}
The nonlinear dependence on \(R\) has been placed in the
\(L^2\)-negligible term \(E_R\).  Dependence on \(R\) in the leading term
\(M_R\) occurs only through the linear functional \(\tr(RW_0)\).

\subsection{Exact characteristic-transform comparison}
\label{sec:proof-transform}

\begin{lemma}[Wishart transform and comparison uniform over \(R\)]
\label{lem:transform}
For every fixed \(t\in\mathbb R\),
\begin{equation}
\frac{\E(e^{itM_R/s_R})}{\E(e^{itM_I/s_R})}
=\exp\left(-\frac{t^2}{2}
\frac{2a/m}{s_R^2}+o(1)\right),
\label{eq:transform-comparison}
\end{equation}
where the \(o(1)\) is uniform over all positive definite correlation
matrices \(R\).
\end{lemma}

\begin{proof}
We first derive the transform.  The density of
\(W_0\sim W_p(m,I_p)\) is
\begin{equation}
f_m(W)
=\frac{|W|^{(m-p-1)/2}\exp(-\tr(W)/2)}
{2^{mp/2}\Gamma_p(m/2)},\qquad W\succ0,
\label{eq:wishart-density}
\end{equation}
Here
\[
\Gamma_p(q)=\pi^{p(p-1)/4}\prod_{j=1}^p
\Gamma\left(q-\frac{j-1}{2}\right).
\]
The density is Theorem~3.2.1 of \cite{Muirhead1982}, and the
matrix gamma integral used below is that book's Theorem~2.1.11.

For a complex number \(\theta\), the definition \eqref{eq:MR} gives
\begin{equation}
e^{\theta M_R}
=\exp(-\theta\E(\log(|W_0|))+p\theta)
 |W_0|^\theta
 \exp\left(-\frac{\theta}{m}\tr(RW_0)\right).
\label{eq:exp-MR}
\end{equation}
Put \(q=m/2+\theta\).  Multiplying \eqref{eq:exp-MR} by
\eqref{eq:wishart-density}, the determinant power becomes
\[
\theta+\frac{m-p-1}{2}=q-\frac{p+1}{2},
\]
and the exponential factor involving \(W\) becomes
\[
\exp\left(-\frac12\tr\left(
\left(I_p+\frac{2\theta}{m}R\right)W\right)\right).
\]
The matrix gamma identity
\[
\int_{W\succ0}|W|^{q-(p+1)/2}
e^{-\tr(CW)/2}\dd W
=2^{pq}\Gamma_p(q)|C|^{-q}
\]
is stated in \cite[Theorem~2.1.11]{Muirhead1982} for complex \(q\) satisfying
\(\Re(q)>(p-1)/2\) and complex symmetric \(C\) whose real part is positive
definite. Lemma~\ref{lem:app-wishart-branches} gives a self-contained
one-variable verification along the curve
\(q(\theta)=m/2+\theta\) and
\(C_R(\theta)=I_p+2\theta R/m\), which is all that is needed here.  It
proves holomorphy in \(\theta\), defines every logarithm branch, and shows
that the continuation includes the imaginary axis.  Thus
\begin{align}
\E(e^{\theta M_R})
={}&e^{-\theta\E(\log(|W_0|))+p\theta}\cdot
2^{p\theta}\cdot\frac{\Gamma_p(m/2+\theta)}
{\Gamma_p(m/2)}
\left|I_p+\frac{2\theta}{m}R\right|^{-q}.
\label{eq:wishart-transform}
\end{align}
where the determinant power is the branch defined in
\eqref{eq:app-determinant-power}.

We shall take \(\theta\) to be imaginary.  Then the right-hand side of
\eqref{eq:wishart-transform} is nonzero: all gamma arguments have positive
real parts, the gamma function has no zeros, and every remaining factor is
an exponential.  Division by the
\(R=I_p\) transform is therefore legitimate.

Write \(R=I_p+A\).  Since
\[
I_p+\frac{2\theta}{m}R
=\left(1+\frac{2\theta}{m}\right)
\left(I_p+\frac{\theta}{q}A\right),
\qquad q=\frac m2+\theta,
\]
the branch compatible cancellation proved in
Lemma~\ref{lem:app-wishart-branches} gives
\begin{equation}
\frac{\E(e^{\theta M_R})}{\E(e^{\theta M_I})}
=\exp\left(-q\,\tr\left(\Log\left(
I_p+\frac{\theta}{q}A\right)\right)\right),
\label{eq:ratio}
\end{equation}
where \(\Log\) is the principal matrix logarithm.  The same appendix lemma
shows that it agrees with the power series logarithm used below once
\(\|(\theta/q)A\|_{\op}<1\).

Take \(\theta=it/s_R\), where \(t\) is fixed, and set
\(X=\theta A/q\) and \(r_n=\|X\|_{\op}\).
Since \(\|A\|_{\op}\le\|A\|_{\F}=\sqrt a\),
\(|q|\ge m/2\), and \(s_R^2\ge2a/m\),
\[
r_n\le\frac{2|t|\sqrt a}{ms_R}
\le\frac{\sqrt2\,|t|}{\sqrt m}\longrightarrow0.
\]

Hence \(r_n\le1/2\) for all sufficiently large \(n\), uniformly in \(R\),
and the matrix logarithm power series converges absolutely. Because \(\tr(A)=0\),
\begin{align}
-q\,\tr(\Log(I_p+X))
=-q\sum_{k=1}^\infty
\frac{(-1)^{k+1}}{k}\tr(X^k) =\frac{\theta^2}{2q}\tr(A^2)+R_{3,n}
=\frac{\theta^2a}{2q}+R_{3,n}.
\label{eq:matrix-log-expansion}
\end{align}
For \(k\ge2\), the eigenvalues of \(A\) give
\begin{equation}
|\tr(A^k)|
\le\sum_{j=1}^p|\lambda_j|^k
\le\|A\|_{\op}^{k-2}\sum_{j=1}^p\lambda_j^2
=a\|A\|_{\op}^{k-2}.
\label{eq:trace-power-bound}
\end{equation}
Using \eqref{eq:trace-power-bound}, \(r_n\le1/2\), and summing a geometric
series,
\begin{align}
|R_{3,n}|
&\le |q|\sum_{k=3}^\infty
\frac{|\theta|^k}{|q|^k}
a\|A\|_{\op}^{k-2} \le C\frac{a|\theta|^3\|A\|_{\op}}{|q|^2}
\le C|t|^3\frac{a^{3/2}}{m^2s_R^3}.
\label{eq:third-order-remainder}
\end{align}
Furthermore, \(2q-m=2\theta\), so
\[
\left|\frac{\theta^2a}{2q}-\frac{\theta^2a}{m}\right|
=|\theta|^2a\left|\frac{m-2q}{2qm}\right|
\le C|t|^3\frac{a}{m^2s_R^3}.
\]
Both errors vanish uniformly. The first satisfies
\[
\frac{a^{3/2}}{m^2s_R^3}
\le\frac{a^{3/2}}{m^2(2a/m)^{3/2}}
=\frac1{2^{3/2}\sqrt m},
\]
with the expression interpreted as zero when \(a=0\). For the second, put
\(\beta=p(p-1)/m^2\) and \(x=2a/m\).
The lower bound \eqref{eq:v-lower} gives \(s_R^2\ge\beta+x\), and hence
\[
\frac{a}{m^2s_R^3}
\le\frac{x}{2m(\beta+x)^{3/2}}.
\]
Differentiation shows that \(x(\beta+x)^{-3/2}\) is maximized at
\(x=2\beta\). Therefore
\[
\frac{a}{m^2s_R^3}
\le\frac{C}{m\sqrt\beta}
=\frac{C}{\sqrt{p(p-1)}}\longrightarrow0.
\]
Finally,
\[
\frac{\theta^2a}{m}
=-\frac{t^2a}{ms_R^2}
=-\frac{t^2}{2}\frac{2a/m}{s_R^2}.
\]
Substitution into \eqref{eq:ratio} proves
\eqref{eq:transform-comparison}.  None of the bounds used an eigenvalue of
\(R\).
\end{proof}

\subsection{The established null CLT, including the hard edge}
\label{sec:proof-null}

\begin{lemma}[Null log-determinant CLT]
\label{lem:null-clt}
If \(p\to\infty\) and \(m\ge p\), then
\begin{equation}
\frac{\log(|\Rhat_0|)-b_{m,p}}{\sqrt{v_{m,p}}}
\Rightarrow N(0,1),
\label{eq:null-clt}
\end{equation}
where \(\Rhat_0\) denotes the sample correlation matrix under \(R=I_p\).
\end{lemma}

\begin{proof}
This is an established null result.  The classical Gram--Schmidt
factorization is
\begin{equation}
|\Rhat_0|\stackrel d=\prod_{j=2}^pB_j,\qquad
B_j\sim\operatorname{Beta}\left(
\frac{m-j+1}{2},\frac{j-1}{2}\right)
\label{eq:null-beta-product}
\end{equation}
with independent factors.  The independent beta product is
\cite[Proposition~2.1(2)]{Rouault2007}; equivalent determinant Mellin
transform formulas appear in \cite[p.~150]{Muirhead1982},
\cite[Lemma~5.10]{JiangYang2013}, and \cite[equation~(1)]{XieSun2021}.
The square endpoint \(m=p\) is included
in \cite[Theorem~3.2(2)]{Rouault2007}.  Taking the first two derivatives
of each beta Mellin transform gives exactly \(b_{m,p}\) and \(v_{m,p}\).
For readability and to verify that no rate of \(m-p\) is omitted,
Appendix~\ref{app:null-clt} gives a complete Lyapunov proof as
Lemma~\ref{lem:app-null-lyapunov}.
\end{proof}

\subsection{Completion via the null reference distribution}
\label{sec:proof-completion}

\begin{proof}[Proof of Theorem~\ref{thm:exact}]
Write
\[
T_R=\log(|\Rhat|)-\E(\log(|\Rhat|)).
\]
For the actual correlation matrix \(R\), the exact identity
\eqref{eq:M-minus-E} says
\begin{equation}
T_R=M_R-E_R.
\label{eq:TR-MR-ER}
\end{equation}
Neither this identity, Lemma~\ref{lem:remainder}, nor
Lemma~\ref{lem:transform} assumes \(R=I_p\).

The matrix \(I_p\) is now introduced only as a reference distribution. If
\(R=I_p\), write
\[
T_I=\log(|\Rhat_0|)-b_{m,p}=M_I-E_I.
\]
Lemma~\ref{lem:null-clt} gives
\[
\frac{T_I}{\sqrt{v_{m,p}}}\Rightarrow N(0,1).
\]

Take an arbitrary subsequence. Since
\[
0\le\alpha_n:=\frac{v_{m,p}}{s_R^2}\le1,
\]
there is a further subsequence along which
\(\alpha_n\to\alpha\in[0,1]\).  On this further subsequence,
\begin{align}
\frac{T_I}{s_R}
&=\frac{\sqrt{v_{m,p}}}{s_R}
\frac{T_I}{\sqrt{v_{m,p}}}
=\sqrt{\alpha_n}\,
\frac{T_I}{\sqrt{v_{m,p}}}
\Rightarrow N(0,\alpha)
\label{eq:null-general-scale}
\end{align}
by Slutsky's theorem.  When \(\alpha=0\), \(N(0,\alpha)\) means the point
mass at zero.

We still need to replace \(T_I\) by \(M_I\). Lemma~\ref{lem:remainder}
applied to \(I_p\) gives
\[
\frac{\E(E_I^2)}{v_{m,p}}\longrightarrow0.
\]
Since \(s_R^2\ge v_{m,p}\),
\[
\frac{\E(E_I^2)}{s_R^2}
\le\frac{\E(E_I^2)}{v_{m,p}}\longrightarrow0.
\]
Thus \(E_I/s_R\to0\) in probability.  From \(M_I=T_I+E_I\) and
\eqref{eq:null-general-scale},
\begin{equation}
\frac{M_I}{s_R}\Rightarrow N(0,\alpha).
\label{eq:MI-general-scale}
\end{equation}

Write
\[
\phi_R(t)=\E(e^{itM_R/s_R}),
\qquad
\phi_I(t)=\E(e^{itM_I/s_R}).
\]
Equation
\eqref{eq:MI-general-scale} gives
\[
\phi_I(t)\longrightarrow e^{-\alpha t^2/2}.
\]
Lemma~\ref{lem:transform} gives
\[
\frac{\phi_R(t)}{\phi_I(t)}
=\exp\left(-\frac{t^2}{2}(1-\alpha_n)+o(1)\right)
\longrightarrow e^{-(1-\alpha)t^2/2},
\]
because \((2a/m)/s_R^2=1-\alpha_n\to1-\alpha\). Therefore the product
of the two limiting factors gives
\[
\phi_R(t)\longrightarrow e^{-t^2/2}.
\]
L\'evy's continuity theorem yields \(M_R/s_R\Rightarrow N(0,1)\).
Finally, \(E_R/s_R\to0\) in probability by Lemma~\ref{lem:remainder};
Slutsky's theorem applied to \eqref{eq:TR-MR-ER} yields
\[
\frac{T_R}{s_R}\Rightarrow N(0,1).
\]
The exact mean in Lemma~\ref{lem:wishart-mean} identifies \(T_R\) with
the numerator of \eqref{eq:main-clt}.

We obtained the desired limit along a further subsequence of an arbitrary
subsequence.  The subsequence principle therefore proves it for the original
sequence \cite[Theorem~2.6, p.~20]{Billingsley1999}.
\end{proof}

\subsection{Equivalence with the elementary normalization}
\label{sec:proof-elementary-normalization}

\begin{lemma}[Equivalence of the two normalizations]
\label{lem:normalization-equivalence}
If \(p\to\infty\) and \(m>p\), then
\begin{equation}
\frac{b_{m,p}-\mu_{0,m,p}}{\sigma_{0,m,p}}\longrightarrow0,
\qquad
\frac{v_{m,p}}{\sigma_{0,m,p}^2}\longrightarrow1.
\label{eq:null-normalization-equivalence}
\end{equation}
Consequently, after adding \(\log(|R|)\) to both centers and \(2a/m\) to
both variances, the resulting standardized statistics differ by \(o_p(1)\).
\end{lemma}

\begin{proof}
Theorem~6 of \cite{JiangQi2015} establishes the null central limit theorem
with the elementary normalization when \(m-p\ge4\). For completeness,
Appendix~\ref{app:normalization} proves directly that the elementary and
exact normalizations are equivalent for every \(m>p\), including the fixed
gaps \(m-p=1,2,3\); see Lemma~\ref{lem:app-normalization}.

It remains only to check the \(R\)-dependent term.  Adding \(\log(|R|)\) to both
centers is exact.  Adding the same nonnegative quantity \(2a/m\) to both
variances preserves their ratio because
\[
\left|
\frac{v_{m,p}+2a/m}{\sigma_{0,m,p}^2+2a/m}-1
\right|
=\frac{|v_{m,p}-\sigma_{0,m,p}^2|}
{\sigma_{0,m,p}^2+2a/m}
\le
\left|\frac{v_{m,p}}{\sigma_{0,m,p}^2}-1\right|.
\]
The centering statement, the variance statement, Slutsky's theorem, and
Theorem~\ref{thm:exact} prove Corollary~\ref{cor:elementary}.
\end{proof}

\section{Scope and discussion}
\label{sec:discussion}

For the determinant likelihood ratio test of complete independence,
Theorem~\ref{thm:exact} supplies the Gaussian shift and variance under any
specified alternative sequence, including \(R_n\to I_p\). Related null
calibrations, asymptotic refinements, and finite sample comparisons appear in
\cite{JiangYang2013,QiWangZhang2019,HuQi2023,BaiZhangLi2024,ParolyaEtAl2024};
we omit power calculations. The determinant is global and does not
localize correlations. Geometrically, the theorem is an all aspect ratio log
volume limit allowing near singular dependence and diverging common factors.

\subsection{Linear spectral statistics beyond the logarithm}

The proof does not cover a general linear spectral statistic. Its immediate
extension is the log affine class \(f(x)=c_0+c_1x+c_2\log(x)\). Since
\(\tr(\Rhat)=p\),
\[
\tr(f(\Rhat))=p(c_0+c_1)+c_2\log(|\Rhat|),
\]
so
Theorem~\ref{thm:exact} applies after deterministic translation and scaling
when \(c_2\ne0\); for \(c_2=0\) the statistic is deterministic. For general
\(f\), the diagonal normalization no longer separates into scalar chi square
terms, and the matrix gamma transform used above no longer retains its
determinant power and linear trace form. Existing resolvent-based linear-spectral-statistic methods typically
work under proportional growth and impose regularity conditions on the
population correlation matrix and the test function
\cite{GaoEtAl2017,MestreVallet2017,YinLiTianZheng2022,
YinZhengZou2023,ChenZhengZou2026}. A general all regime theorem would require a different
argument.

\subsection{Limits of the extension beyond Gaussianity}

Corollary~\ref{cor:matrix-spherical} is exact because a common matrix radius
cancels. It does not cover independent non-Gaussian observations: ordinary
elliptical samples have observation specific radii and unequal weights.
Indeed, \cite{ParolyaEtAl2024} has a fourth moment mean correction, and
independent component and elliptical linear spectral statistic CLTs can differ
even at the identity \cite{YinZhengZou2023}. Thus, outside Gaussian or the
matrix spherical class, \(R_n\) alone does not in general determine the
present centering. Uncentered null
universality is broader
\cite{HeinyParolya2024,LiLiuXieZhou2026}, but it does not supply the centered,
arbitrary correlation matrix theorem considered here.

Within the independent Gaussian Pearson model, Theorem~\ref{thm:exact} covers
every \(p\to\infty\), \(n-1\ge p\), and positive definite \(R_n\). It extends
proportional or spectrally regular theorems to the nonsingular dimensional
domain, without superseding fixed dimensional or non-Gaussian results. The
second chaos projection linearizes the diagonal normalization, while the
Wishart transform handles the \(R\)-dependent term and null determinant jointly.

Close to identity matrices are not exceptional. When \(R_n\to I_p\), the
\(R_n\)-dependent variance can vanish, but the null variance remains and the
comparison reduces to the null law. When an eigenvalue vanishes or a spike
grows, the Frobenius term controls the matrix logarithm expansion.

Three directions remain open. First, an unrestricted non-Gaussian alternative theorem would need a substitute for both exact Gaussian chaos orthogonality and the Wishart transform. Second, a uniform normal approximation rate would turn the limiting theorem into quantified finite sample guarantees. Third, joint limit theory with entrywise or spectral statistics could combine sensitivity to global redundancy with the ability to localize departures.

\appendix

\section{Branch choices in the Wishart transform}
\label{app:wishart-branches}

For \(z\ne0\), the power \(z^\alpha=\exp(\alpha\log(z))\) is ambiguous because
changing \(\log(z)\) by \(2\pi i k\) generally changes the power.  This affects
\(\left|C_R(\theta)\right|^{-q(\theta)}\), where the vertical bars denote a
determinant, not a modulus.  Nor can one use the principal scalar logarithm of
the determinant.  For example, if
\(m=p=4\), \(R=I_4\), and \(\theta=2i\), then
\(C_R(\theta)=(1+i)I_4\), whose determinant is \(-4\), on the scalar
principal logarithm branch cut, although all four eigenvalues lie in the open
right half plane.  We therefore use the trace of the principal matrix
logarithm.  The affine function \(q(\theta)\) needs no branch; \(\Gamma_p\), the multivariate gamma function, is
single valued and pole free below; \(2^{pq(\theta)}\) uses the real logarithm
of \(2\); and powers of \(|W|\), for \(W\succ0\), use the real logarithm of
the positive determinant.  The next lemma verifies all remaining branches.

\begin{lemma}[One parameter branch verification]
\label{lem:app-wishart-branches}
Fix integers \(m\ge p\ge1\), a positive definite correlation matrix \(R\),
and write \(A=R-I_p\). Define
\[
q(\theta)=\frac m2+\theta,
\qquad c(\theta)=1+\frac{2\theta}{m},
\qquad C_R(\theta)=I_p+\frac{2\theta}{m}R,
\qquad B_R(\theta)=I_p+\frac{\theta}{q(\theta)}A.
\]
Let
\[
\delta_{m,p}=\min\left(\frac{m-p+1}{2},\frac{m}{2p}\right),
\qquad
\mathcal H_{m,p}
=\{\theta\in\mathbb C:\Re(\theta)>-\delta_{m,p}\}.
\]
Then \(\mathcal H_{m,p}\) is a connected open half-plane containing the
imaginary axis.  On this half-plane,
\(\Re(q(\theta))>(p-1)/2\), the spectrum of \(C_R(\theta)\) lies in the open
right half-plane, and the spectrum of \(B_R(\theta)\) avoids
\((-\infty,0]\).  With \(\Log\) denoting the principal scalar or matrix
logarithm, define
\begin{equation}
\left|C_R(\theta)\right|^{-q(\theta)}
:=\exp\left(-q(\theta)\tr(\Log(C_R(\theta)))\right).
\label{eq:app-determinant-power}
\end{equation}
This is a single-valued, nonzero holomorphic function of \(\theta\) on
\(\mathcal H_{m,p}\), and
\begin{align}
&\int_{W\succ0}|W|^{q(\theta)-(p+1)/2}
 \exp\left(-\frac12\tr(C_R(\theta)W)\right)\dd W\notag\\
&\qquad=2^{pq(\theta)}\Gamma_p(q(\theta))
 \exp\left(-q(\theta)\tr(\Log(C_R(\theta)))\right).
\label{eq:app-complex-matrix-gamma}
\end{align}
Moreover, the factorization \(C_R(\theta)=c(\theta)B_R(\theta)\) is branch
compatible in the sense that
\begin{equation}
\tr(\Log(C_R(\theta)))
=p\Log(c(\theta))+\tr(\Log(B_R(\theta))).
\label{eq:app-log-factorization}
\end{equation}
Consequently, since \(C_I(\theta)=c(\theta)I_p\),
\[
\frac{|C_R(\theta)|^{-q(\theta)}}
{|C_I(\theta)|^{-q(\theta)}}
=
\exp\left(
-q(\theta)\tr\left(\Log\left(B_R(\theta)\right)\right)
\right).
\]
Here both determinant powers are defined by
\eqref{eq:app-determinant-power}.
Finally, whenever \(\|\theta A/q(\theta)\|_{\op}<1\), the operator norm
convergent identity
\[
\Log(B_R(\theta))
=\sum_{k=1}^\infty\frac{(-1)^{k+1}}{k}
\left(\frac{\theta A}{q(\theta)}\right)^k
\]
holds.
\end{lemma}

\begin{proof}
Let \(r_1,\ldots,r_p\) be the eigenvalues of \(R\), and choose an orthogonal
matrix \(O\) such that \(R=O\diag(r_1,\ldots,r_p)O^\top\). The same
diagonalization applies to every matrix function of \(R\) used below.
For a matrix \(M\) whose spectrum avoids \((-\infty,0]\), the principal
matrix logarithm is defined by the holomorphic functional calculus
\cite[Chapter~11]{Higham2008}. In particular, if
\(\lambda_1,\ldots,\lambda_p\) are the eigenvalues of \(M\), counted with
multiplicity, then
\[
\tr(\Log(M))=\sum_{j=1}^p\Log(\lambda_j).
\]

Since
\(\tr(R)=p\), \(0<r_j\le p\). For \(\theta\in\mathcal H_{m,p}\), the two
bounds in the definition of \(\delta_{m,p}\) give
\(\Re(q(\theta))>(p-1)/2\) and \(\Re(m/2+r_j\theta)>0\).  Hence every
eigenvalue \(1+2r_j\theta/m\) of \(C_R(\theta)\) lies in the open right
half-plane. Its principal matrix logarithm is therefore well defined. The
fixed diagonalization
\[
\Log(C_R(\theta))
=O\diag\left(
\Log\left(1+\frac{2r_1\theta}{m}\right),\ldots,
\Log\left(1+\frac{2r_p\theta}{m}\right)
\right)O^\top.
\]
This shows directly that it is holomorphic in
\(\theta\). Consequently,
\eqref{eq:app-determinant-power} defines a single-valued, nonzero
holomorphic function of \(\theta\) on \(\mathcal H_{m,p}\).

We prove \eqref{eq:app-complex-matrix-gamma} by continuation only in
\(\theta\).  For \(W\succ0\), the power in its integrand means
\(\exp((q(\theta)-(p+1)/2)\log(|W|))\), with a real logarithm.  Fix compact
\(K\subset\mathcal H_{m,p}\), and let \(q_-\) and \(q_+\) be the infimum
and supremum of \(\Re(q(\theta))\) on \(K\), while
\(u_K=\inf_{\theta\in K}\Re(\theta)\).  Then \(q_->(p-1)/2\), and the
Hermitian part \(I_p+2\Re(\theta)R/m\) of \(C_R(\theta)\) is at least
\(\varepsilon_KI_p\), where \(\varepsilon_K=1\) if \(u_K\ge0\) and
\(\varepsilon_K=1+2pu_K/m>0\) otherwise.  Here we used
\(u_K>-m/(2p)\) and \(0\prec R\preceq pI_p\).  Thus the integrand's modulus
is bounded by
\[
\left\{|W|^{q_--(p+1)/2}+|W|^{q_+-(p+1)/2}\right\}
\exp\left(-\frac{\varepsilon_K}{2}\tr(W)\right).
\]
Both terms are integrable real matrix gamma kernels.  Pointwise holomorphy
and this local domination show that the integral is holomorphic in \(\theta\),
for example by Morera's theorem.  Its proposed value is also holomorphic:
every gamma argument has positive real part and the matrix logarithm is
holomorphic.  For real \(\theta\in\mathcal H_{m,p}\), the equality is the
real matrix gamma identity \cite[Theorem~2.1.11]{Muirhead1982}.  The
one variable identity theorem proves it on the connected half-plane.

For the factorization, put \(n_j(\theta)=m/2+r_j\theta\).  Both
\(n_j(\theta)\) and \(q(\theta)\) lie in the open right half-plane, and the
eigenvalues of \(B_R(\theta)\) are
\(b_j(\theta)=n_j(\theta)/q(\theta)\).  The two principal arguments lie in
\((-\pi/2,\pi/2)\), so their difference lies in \((-\pi,\pi)\).  Therefore
\(b_j(\theta)\notin(-\infty,0]\). Moreover,
\(1+2r_j\theta/m=(2/m)n_j(\theta)\) and
\(c(\theta)=(2/m)q(\theta)\), so
\[
\begin{aligned}
\Log(b_j(\theta))
&=\Log(n_j(\theta))-\Log(q(\theta)),\\
\Log\left(1+\frac{2r_j\theta}{m}\right)
&=\Log(c(\theta))+\Log(b_j(\theta)).
\end{aligned}
\]
Summing over
\(j\) proves \eqref{eq:app-log-factorization} without any unrecorded
multiple of \(2\pi i\).

Finally, if \(\|\theta A/q(\theta)\|_{\op}<1\), the usual logarithm power
series converges absolutely in operator norm.  On every eigenvalue it equals
the principal scalar logarithm, so it equals the principal matrix logarithm
and proves the stated series identity.
\end{proof}

\section{The Hermite covariance identity}
\label{app:mehler-hermite}

This appendix records the precise consequence of Mehler's formula used in
\eqref{eq:mehler-hermite}.  The one-dimensional correlated-Hermite identity
is \cite[Proposition~2.2.1, p.~26]{NourdinPeccati2012}; the
Ornstein--Uhlenbeck semigroup and Mehler's formula are given by
\cite[Definition~2.8.1 and Theorem~2.8.2,
pp.~45--47]{NourdinPeccati2012}.  The elementary Gaussian calculation below
makes the multivariate identity and the cases \(\rho<0\) and
\(\rho=\pm1\) explicit.

\begin{lemma}[Hermite covariance consequence of Mehler's formula]
\label{lem:app-mehler-hermite}
Let \(G\) and \(G'\) be jointly Gaussian random vectors such that
\[
G,G'\sim N_m(0,I_m),
\qquad \Cov(G,G')=\rho I_m,
\qquad |\rho|\le1.
\]
For every \(\beta\in\mathbb N_0^m\),
\begin{equation}
\E(H_\beta(G')\mid G)
=\rho^{|\beta|}H_\beta(G)
\quad\text{almost surely}.
\label{eq:app-mehler-conditional}
\end{equation}
Consequently, for every
\(\alpha,\beta\in\mathbb N_0^m\),
\begin{equation}
\E(H_\alpha(G)H_\beta(G'))
=\1_{(\alpha=\beta)}\alpha!\rho^{|\alpha|}.
\label{eq:app-mehler-covariance}
\end{equation}
Here and below, \(0^0=1\).
\end{lemma}

\begin{proof}
First suppose that \(|\rho|<1\), and define
\[
Z=\frac{G'-\rho G}{\sqrt{1-\rho^2}}.
\]
Because \(G\) and \(G'\) are jointly Gaussian, the pair \((G,Z)\) is also
jointly Gaussian. Direct covariance calculations give
\begin{align*}
\Cov(Z,G)
&=\frac{\Cov(G',G)-\rho\Cov(G,G)}{\sqrt{1-\rho^2}}=0,\\
\Cov(Z,Z)
&=\frac{\Cov(G',G')-\rho\Cov(G',G)
        -\rho\Cov(G,G')+\rho^2\Cov(G,G)}
        {1-\rho^2}\\
&=\frac{I_m-\rho^2I_m-\rho^2I_m+\rho^2I_m}
        {1-\rho^2}
=I_m.
\end{align*}
Thus \(Z\sim N_m(0,I_m)\).  Moreover, jointly Gaussian random vectors with
zero cross-covariance are independent, so \(Z\) is independent of \(G\).
Equivalently,
\begin{equation}
G'=\rho G+\sqrt{1-\rho^2}\,Z,
\qquad Z\sim N_m(0,I_m),
\qquad Z\ \text{independent of }G.
\label{eq:app-mehler-coupling}
\end{equation}

We next compute one coordinate at a time.  The generating function for the
probabilists' Hermite polynomials is
\begin{equation}
\exp(tx-t^2/2)
=\sum_{k=0}^\infty H_k(x)\frac{t^k}{k!}.
\label{eq:app-hermite-generating}
\end{equation}
For \(j\in\{1,\ldots,m\}\), equation
\eqref{eq:app-mehler-coupling} and the moment-generating function of
\(Z_j\sim N(0,1)\) give
\begin{align}
\E\left(\left.
  \exp(tG'_j-t^2/2)\,\right|\,G\right)
&=\exp(t\rho G_j-t^2/2)
  \E\left(\exp\left(t\sqrt{1-\rho^2}\,Z_j\right)\right)\notag\\
&=\exp(t\rho G_j-\rho^2t^2/2) =\sum_{k=0}^\infty
  \rho^kH_k(G_j)\frac{t^k}{k!}.
\label{eq:app-mehler-coordinate-generating}
\end{align}
On the other hand, differentiating the left side \(k\) times at \(t=0\)
and using \eqref{eq:app-hermite-generating} gives
\(\E(H_k(G'_j)\mid G)\).  This differentiation under the conditional
expectation is valid because a Gaussian random variable has finite
exponential moments in a neighborhood of \(0\).  Comparing the derivatives
at \(0\) on the two sides of
\eqref{eq:app-mehler-coordinate-generating} therefore yields
\[
\E(H_k(G'_j)\mid G)=\rho^kH_k(G_j).
\]

Conditional on \(G\), the coordinates \(G'_1,\ldots,G'_m\) are independent:
by \eqref{eq:app-mehler-coupling}, their remaining randomness comes from the
independent coordinates of \(Z\). Hence
\[
\begin{aligned}
\E(H_\beta(G')\mid G)
&=\prod_{j=1}^m\E(H_{\beta_j}(G'_j)\mid G)\\
&=\prod_{j=1}^m\rho^{\beta_j}H_{\beta_j}(G_j)
=\rho^{|\beta|}H_\beta(G),
\end{aligned}
\]
which proves \eqref{eq:app-mehler-conditional} when \(|\rho|<1\).

It remains to check \(\rho=\pm1\). In either case,
\[
\Cov(G'-\rho G,G'-\rho G)=(1-\rho^2)I_m=0,
\]
and therefore
\(G'=\rho G\) almost surely. The parity relation
\[
H_k(-x)=(-1)^kH_k(x)
\]
follows immediately by replacing \(t\) with \(-t\) in
\eqref{eq:app-hermite-generating}. Thus
\[
H_\beta(\rho G)=\rho^{|\beta|}H_\beta(G),
\qquad \rho\in\{-1,1\},
\]
which proves \eqref{eq:app-mehler-conditional} at both endpoints.

Finally, the tower property and the Hermite orthogonality relation
\eqref{eq:hermite-orthogonality} give
\begin{align*}
\E(H_\alpha(G)H_\beta(G'))
&=\E\left(
  H_\alpha(G)\E(H_\beta(G')\mid G)\right) =\rho^{|\beta|}
  \E(H_\alpha(G)H_\beta(G)) =\1_{(\alpha=\beta)}\alpha!\rho^{|\alpha|}.
\end{align*}
This is \eqref{eq:app-mehler-covariance}.
\end{proof}

\section{Exact finite-sample variance}
\label{app:exact-variance}

This appendix proves Proposition~\ref{prop:exact-cumulants}.  The Wishart
and Kibble bivariate-gamma distributions are classical; the calculations
below differentiate those cited identities and assemble the covariances
needed for the sample-correlation log determinant.

\begin{lemma}[Determinant--diagonal cross-covariance]
\label{lem:app-crosscov}
For every \(i\),
\begin{equation}
\Cov(\log(|W_0|),\log(Q_i))=\psi_1(m/2).
\label{eq:app-crosscov}
\end{equation}
\end{lemma}

\begin{proof}
Recall from \eqref{eq:gi-definition} that
\(Q_i=v_i^\top W_0v_i\), where \(\|v_i\|=1\).  Choose an orthogonal matrix
\(O\) whose first row is \(v_i^\top\).  Orthogonal invariance of
\(W_0\sim W_p(m,I_p)\) gives
\[
OW_0O^\top\stackrel d=W_0,
\]
whose
upper-left entry is \(Q_i\). For a partitioned Wishart matrix, the
Schur-complement decomposition gives
\[
|OW_0O^\top|=Q_i|W_{22\cdot1}|,
\qquad Q_i\ \text{is independent of }W_{22\cdot1},
\]
by \cite[Theorem~3.2.10]{Muirhead1982}. Because an orthogonal
transformation does not change the determinant,
\[
\log(|W_0|)=\log(Q_i)+\log(|W_{22\cdot1}|).
\]
Taking covariance with \(\log(Q_i)\) gives
\[
\Cov(\log(|W_0|),\log(Q_i))
=\Var(\log(Q_i))=\psi_1(m/2).
\]
Since \(Q_i\sim\chi_m^2\), differentiating its gamma Mellin transform twice
gives the last equality, as in the proof of
Lemma~\ref{lem:second-chaos}.
\end{proof}

\begin{lemma}[Kibble covariance series]
\label{lem:app-kibble}
If \(i\ne j\), then
\begin{equation}
\Cov(\log(Q_i),\log(Q_j))
=c_m(r_{ij})
=\sum_{k=1}^\infty
\frac{(k-1)!}{k(m/2)_k}r_{ij}^{2k}.
\label{eq:app-cm}
\end{equation}
Furthermore, for every \(|r|\le1\),
\begin{equation}
0\le c_m(r)-\frac{2r^2}{m}
\le\frac{4r^4}{m^2}.
\label{eq:app-cm-bound}
\end{equation}
\end{lemma}

\begin{proof}
Put \(\alpha=m/2\) and \(r=r_{ij}\).  By
\eqref{eq:gi-covariance}, the pair \((Q_i,Q_j)\) is the pair of squared
norms of two standard \(m\)-Gaussian vectors with coordinatewise
correlation \(r\).  Its Kibble bivariate-gamma product-moment formula is
\begin{equation}
\E(Q_i^sQ_j^t)
=2^{s+t}
\frac{\Gamma(\alpha+s)\Gamma(\alpha+t)}{\Gamma(\alpha)^2}
\,{}_2F_1(-s,-t;\alpha;r^2);
\label{eq:kibble-mellin}
\end{equation}
see \cite{Kibble1941,NadarajahKotz2006}. It follows by integrating the
Kibble density term by term; the latter reference derives the corresponding
general product moments. Because \(R\) is positive
definite, \(|r|<1\) when \(i\ne j\).  The defining hypergeometric series
\[
{}_2F_1(-s,-t;\alpha;r^2)
=\sum_{k=0}^\infty
\frac{(-s)_k(-t)_k}{(\alpha)_k\,k!}r^{2k}
\]
is locally uniformly convergent near \((s,t)=(0,0)\), so it may be
differentiated term by term. For \(k\ge1\),
\(\left.\partial_s(-s)_k\right|_{s=0}=-(k-1)!\), and the analogous
identity holds for \(t\). The gamma factors in
\eqref{eq:kibble-mellin} separate into a function of \(s\) and a function
of \(t\); they therefore contribute no mixed derivative to the logarithm
of the moment-generating function.  The first partial derivatives of the
hypergeometric factor vanish at the origin, so the mixed derivative of its
logarithm equals the mixed derivative of the factor itself.  Hence
\[
\Cov(\log(Q_i),\log(Q_j))
=\sum_{k=1}^\infty
\frac{(k-1)!^2}{(\alpha)_k\,k!}r^{2k}
=\sum_{k=1}^\infty
\frac{(k-1)!}{k(\alpha)_k}r^{2k},
\]
which proves \eqref{eq:app-cm}.

The \(k=1\) term is \(r^2/\alpha=2r^2/m\), and all coefficients are
nonnegative. We now justify the endpoints in the definition of \(c_m\).
Let \(G,Z\) be independent standard \(m\)-variate Gaussian vectors and set
\[
G_r=rG+\sqrt{1-r^2}\,Z,
\qquad 0\le r<1.
\]
Then \(G_r\to G\) almost surely as \(r\uparrow1\), and
\(\log(\|G_r\|^2)\to\log(\|G\|^2)\) in \(L^2\). Indeed, convergence holds in
probability, all variables \(\log(\|G_r\|^2)\) have the same square
integrable log chi square law, and their squared differences are uniformly
integrable. Hence the covariance tends to
\[
c_m(1)=\Var(\log(\chi_m^2))=\psi_1(\alpha).
\]
Monotone convergence of the nonnegative series gives the same endpoint
value. Since the series is even in \(r\), \(c_m(-1)=c_m(1)\).

The \(k=1\) term in the defining series for \(c_m(r)\) is
\(r^2/\alpha\). Since all the coefficients are nonnegative and
\(r^{2k}\le r^4\) for \(|r|\le1\) and \(k\ge2\), we obtain
\begin{align*}
0
&\le c_m(r)-\frac{r^2}{\alpha}
=\sum_{k=2}^{\infty}
  \frac{(k-1)!}{k(\alpha)_k}r^{2k}\le r^4\sum_{k=2}^{\infty}
  \frac{(k-1)!}{k(\alpha)_k} =r^4\left\{c_m(1)-\frac1\alpha\right\} =r^4\left\{\psi_1(\alpha)-\frac1\alpha\right\}.
\end{align*}
The trigamma bound proved in Lemma~\ref{lem:second-chaos} gives
\(\psi_1(\alpha)-1/\alpha\le1/\alpha^2=4/m^2\), proving
\eqref{eq:app-cm-bound}.
\end{proof}

\begin{lemma}[Assembly of the exact variance]
\label{lem:app-exact-variance}
The variance and comparison in
\eqref{eq:exact-var}--\eqref{eq:variance-comparison} hold.
\end{lemma}

\begin{proof}
Write
\[
S_{m,p}=\sum_{j=1}^p\psi_1\left(\frac{m-j+1}{2}\right).
\]
Bartlett's decomposition gives
\[
\Var(\log|W_0|)=S_{m,p}.
\]
By Lemma~\ref{lem:app-kibble},
\[
\Var\left(\sum_i\log Q_i\right)
=p\psi_1(m/2)+\sum_{i\ne j}c_m(r_{ij}).
\]
Lemma~\ref{lem:app-crosscov} gives
\[
\Cov\left(\log|W_0|,\sum_i\log Q_i\right)=p\psi_1(m/2).
\]
Taking the variance in the exact determinant identity
\eqref{eq:wishart-decomp} and substituting these three expressions,
\begin{align*}
\Var(\log(|\Rhat|))
&=S_{m,p}
+p\psi_1(m/2)+\sum_{i\ne j}c_m(r_{ij})
-2p\psi_1(m/2)\\
&=S_{m,p}-p\psi_1(m/2)
+\sum_{i\ne j}c_m(r_{ij}) =v_{m,p}+\sum_{i\ne j}c_m(r_{ij}).
\end{align*}
The final equality is the definition \eqref{eq:vmp}, with its missing
\(j=1\) term made explicit.

Because \(A=R-I_p\) has off-diagonal entries \(r_{ij}\) and zero diagonal,
\[
\sum_{i\ne j}r_{ij}^2=\tr(A^2)=a,
\qquad
\sum_{i\ne j}r_{ij}^4\le a.
\]
Summing \eqref{eq:app-cm-bound} over the ordered pairs \(i\ne j\)
and using these two identities gives
\begin{align*}
0
&\le
\Var(\log|\widehat R|)
-\left(v_{m,p}+\frac{2a}{m}\right) =
\sum_{i\ne j}
\left\{c_m(r_{ij})-\frac{2r_{ij}^2}{m}\right\} \le
\frac{4}{m^2}\sum_{i\ne j}r_{ij}^4
\le
\frac{4a}{m^2}.
\end{align*}
If \(a>0\), division by
\(v_{m,p}+2a/m\ge2a/m\) makes the relative error at most \(2/m\).
If \(a=0\), the error is exactly zero.  This proves
Proposition~\ref{prop:exact-cumulants}, including the claim that the exact
scale may replace \(s_R\).
\end{proof}

\section{A self-contained Lyapunov proof of the null CLT}
\label{app:null-clt}

The null CLT is cited in Lemma~\ref{lem:null-clt}.  This appendix supplies
a complete proof for readers who want to verify directly that the beta
product covers every sequence \(m\ge p\), including \(m=p\).

\begin{lemma}[Lyapunov verification for the null beta product]
\label{lem:app-null-lyapunov}
The convergence in \eqref{eq:null-clt} holds whenever \(p\to\infty\) and
\(m\ge p\).
\end{lemma}

\begin{proof}
Use the independent beta factors in \eqref{eq:null-beta-product}. Set
\[
a_j=\frac{m-j+1}{2},
\qquad
b_j=\frac{j-1}{2},
\qquad
a_j+b_j=\frac m2,
\]
and
write \(L_j=\log(B_j)\).  The beta Mellin transform gives
\begin{align}
K_j(t):=\log(\E(e^{tL_j}))
= \log(\Gamma(a_j+t))-\log(\Gamma(a_j)) +\log(\Gamma(m/2))-\log(\Gamma(m/2+t)).
\label{eq:app-beta-cgf}
\end{align}
Differentiating at zero,
\begin{align}
\E(L_j)
&=\psi(a_j)-\psi(m/2),\notag\\
w_j:=\Var(L_j)
&=\psi_1(a_j)-\psi_1(m/2),\notag\\
d_j:=\operatorname{cum}_4(L_j)
&=\psi_3(a_j)-\psi_3(m/2).
\label{eq:app-beta-cumulants}
\end{align}
Summing the first two lines gives \(b_{m,p}\) and \(v_{m,p}\).

We shall repeatedly use the elementary consequence of the polygamma series:
for every fixed integer \(\ell\ge1\), there is a constant \(C_\ell\) such
that
\begin{equation}
|\psi_\ell(x)|
\le C_\ell(x^{-\ell}+x^{-\ell-1}),
\qquad x>0.
\label{eq:app-polygamma-bound}
\end{equation}
Indeed, repeated termwise differentiation of the locally uniformly
convergent trigamma series in
\cite[Eq.~(5.15.1), p.~144]{OlverEtAl2010} gives
\begin{equation}
|\psi_\ell(x)|
=\ell!\sum_{k=0}^\infty(x+k)^{-\ell-1},
\qquad x>0.
\label{eq:app-polygamma-series}
\end{equation}
Here the sign of \(\psi_\ell(x)\) is \((-1)^{\ell+1}\).
The decreasing sum is bounded by its first term plus the integral of
the same function, which proves \eqref{eq:app-polygamma-bound}.

Put \(Y_j=L_j-\E(L_j)\). We verify Lyapunov's condition with exponent four:
\[
\frac{\sum_{j=2}^p\E(|Y_j|^4)}{v_{m,p}^2}\longrightarrow0.
\]
There are three regimes.

\smallskip
\noindent\emph{Regime 1: \(p\le m/2\).}
Here \(a_j\ge(m-p+1)/2\ge m/4\).  Apply the mean-value theorem to the
second and fourth cumulants in \eqref{eq:app-beta-cumulants}.  The interval
length is \((j-1)/2\), while \eqref{eq:app-polygamma-bound} bounds
\(-\psi_2(x)\) by \(C/m^2\) and \(-\psi_4(x)\) in absolute value by
\(C/m^4\) on this interval.  Hence
\[
\begin{aligned}
w_j&\le\frac{C(j-1)}{m^2},
&0\le d_j&\le\frac{C(j-1)}{m^4},\\
\max_jw_j&\le\frac{Cp}{m^2},
&\sum_{j=2}^pd_j&\le\frac{Cp^2}{m^4},\\
\frac{\max_jw_j}{v_{m,p}}&=O(p^{-1}),
&\frac{\sum_jd_j}{v_{m,p}^2}&=O(p^{-2}).
\end{aligned}
\]

\smallskip
\noindent\emph{Regime 2: \(p>m/2\) and \(d=m-p\to\infty\).}
The numbers \(2a_j=m-j+1\) range from \(m-1\) down to \(d+1\).
The bound \eqref{eq:app-polygamma-bound} gives
\[
\sum_{j=2}^pd_j
\le C\sum_{k=d+1}^{m-1}k^{-3}
\le\frac{C}{(d+1)^2}\longrightarrow0,
\qquad
\max_jw_j
\le\psi_1\left(\frac{d+1}{2}\right)
\le\frac{C}{d+1}\longrightarrow0.
\]
In this regime, \eqref{eq:v-lower} is bounded below by a positive constant
because \(p/m>1/2\).  Therefore \(\max_jw_j/v_{m,p}\to0\) and
\(\sum_jd_j/v_{m,p}^2\to0\).

\smallskip
\noindent\emph{Regime 3: \(d=m-p\) remains bounded.}
The same estimates give
\[
\sum_{j=2}^pd_j=O(1),
\qquad
\max_jw_j=O(1).
\]
We need the sharper fact \(v_{m,p}\to\infty\).  Reindexing
\eqref{eq:vmp} by \(k=m-j+1\),
\[
v_{m,p}
=\sum_{k=d+1}^{m-1}\psi_1(k/2)
-(p-1)\psi_1(m/2).
\]
The trigamma series and integral comparison give
\[
\psi_1(k/2)\ge\frac2k,
\qquad
\psi_1(m/2)\le\frac2m+\frac4{m^2}.
\]
Consequently,
\[
v_{m,p}
\ge2\sum_{k=d+1}^{m-1}\frac1k
-(p-1)\left(\frac2m+\frac4{m^2}\right)
\ge c\log(m)
\]
for all sufficiently large \(m\), where \(c>0\) may depend on a bound for
\(d\).  Thus \(\max_jw_j/v_{m,p}\to0\) and
\(\sum_jd_j/v_{m,p}^2\to0\).

It remains to verify Lyapunov's condition.  The moment--cumulant identity
gives
\[
\E(Y_j^4)=d_j+3w_j^2,
\qquad
\sum_{j=2}^pw_j^2
\le(\max_jw_j)\sum_{j=2}^pw_j
=(\max_jw_j)v_{m,p}.
\]
Therefore, in all three regimes,
\[
\frac{\sum_{j=2}^p\E(Y_j^4)}{v_{m,p}^2}
\le\frac{\sum_jd_j}{v_{m,p}^2}
+3\frac{\max_jw_j}{v_{m,p}}\longrightarrow0.
\]
The Lyapunov triangular-array CLT
\cite[Theorem~4.9, p.~126]{Petrov1995} yields
\[
\frac{\sum_{j=2}^pY_j}{\sqrt{v_{m,p}}}\Rightarrow N(0,1).
\]
Since \(\sum_jY_j=\log(|\Rhat_0|)-b_{m,p}\), this is
\eqref{eq:null-clt}.  If a sequence does not remain in one regime, every
subsequence has a further subsequence in one of the three; the subsequence
principle completes the proof.
\end{proof}

\section{Equivalence with the elementary normalization}
\label{app:normalization}

\begin{lemma}[Uniform normalization comparison for every \(m>p\)]
\label{lem:app-normalization}
If \(p\to\infty\) and \(m>p\), then
\[
\frac{b_{m,p}-\mu_{0,m,p}}{\sigma_{0,m,p}}\longrightarrow0,
\qquad
\frac{v_{m,p}}{\sigma_{0,m,p}^2}\longrightarrow1.
\]
\end{lemma}

\begin{proof}
Put
\[
d=m-p\ge1,
\qquad y=\frac pm,
\qquad L=\log\left(\frac md\right).
\]
Substituting \(1-p/m=d/m\) into
\eqref{eq:elementary-null-mean}--\eqref{eq:elementary-null-var} gives the exact
rewriting
\begin{equation}
\mu_{0,m,p}
=\left(d-\frac12\right)L-p+\frac pm,
\qquad
\sigma_{0,m,p}^2=2\left(L-\frac pm\right).
\label{eq:app-elementary-rewrite}
\end{equation}

The digamma and trigamma expansions in
\cite[Eq.~(5.11.2), p.~140, and Eq.~(5.15.8),
p.~144]{OlverEtAl2010} give, as \(k\to\infty\),
\begin{equation}
\psi(k/2)=\log(k/2)-\frac1k+O(k^{-2}),
\qquad
\psi_1(k/2)=\frac2k+O(k^{-2}).
\label{eq:app-polygamma-expansions}
\end{equation}
After enlarging the remainder constants over a finite initial range, the
bounds are uniform for every integer \(k\ge1\).  Reindex
\eqref{eq:bmp}--\eqref{eq:vmp} by \(k=m-j+1\).  Summing
\eqref{eq:app-polygamma-expansions} gives
\begin{align}
b_{m,p}
={}&\log\left(\frac{\Gamma(m)}{\Gamma(d+1)m^{p-1}}\right)
-(H_{m-1}-H_d)+\frac{p-1}{m} +O\left(\sum_{k=d+1}^{m-1}k^{-2}+\frac p{m^2}\right),
\label{eq:app-b-sum}\\
v_{m,p}
={}&2(H_{m-1}-H_d)-\frac{2(p-1)}m +O\left(\sum_{k=d+1}^{m-1}k^{-2}+\frac p{m^2}\right).
\label{eq:app-v-sum}
\end{align}
Here \(H_r=\sum_{k=1}^r1/k\).  The logarithms in the digamma sum form the
gamma ratio, the \(1/k\) terms form the harmonic-number difference, and the
displayed error is the sum of the pointwise remainders.

We now make the Stirling remainder uniform, including fixed \(d\).
The positive real error bound for Stirling's series
\cite[Section~5.11(ii), pp.~140--141]{OlverEtAl2010} implies that a
universal constant \(C\) exists such that, for all integers \(m\ge2\) and
\(d\ge1\),
\begin{align}
\log(\Gamma(m))
&=\left(m-\frac12\right)\log(m)-m
  +\frac12\log(2\pi)+r_m,
& |r_m|&\le\frac{C}{m},
\label{eq:uniform-stirling-m}\\
\log(\Gamma(d+1))
&=\left(d+\frac12\right)\log(d)-d
  +\frac12\log(2\pi)+\widetilde r_d,
& |\widetilde r_d|&\le\frac{C}{d}.
\label{eq:uniform-stirling-d}
\end{align}
The second line is the factorial form, obtained from the first line and
\(\Gamma(d+1)=d\Gamma(d)\). Since \(p=m-d\), subtraction gives
\begin{equation}
\log\left(\frac{\Gamma(m)}
{\Gamma(d+1)m^{p-1}}\right)
=\left(d+\frac12\right)L-p+O(d^{-1}),
\label{eq:uniform-gamma-ratio}
\end{equation}
with a universal implied constant.

The harmonic and quadratic remainders are also uniform. Integral comparison
gives
\[
0\le L-(H_{m-1}-H_d)
\le\log\left(1+\frac1d\right)\le\frac1d,
\qquad
\sum_{k=d+1}^{m-1}k^{-2}+\frac p{m^2}
\le\frac1d+\frac1m\le\frac2d.
\]
Substitution of these bounds and \eqref{eq:uniform-gamma-ratio} into
\eqref{eq:app-b-sum}--\eqref{eq:app-v-sum} yields
\[
\begin{aligned}
b_{m,p}
&=\left(d-\frac12\right)L-p+\frac{p-1}{m}+O(d^{-1})
=\mu_{0,m,p}+O(d^{-1}),\\
v_{m,p}
&=2L-\frac{2(p-1)}m+O(d^{-1})
=\sigma_{0,m,p}^2+O(d^{-1}).
\end{aligned}
\]
Thus, for a universal constant \(C\),
\begin{equation}
|b_{m,p}-\mu_{0,m,p}|\le\frac Cd,
\qquad
|v_{m,p}-\sigma_{0,m,p}^2|\le\frac Cd
\label{eq:app-coarse-equivalence}
\end{equation}
for every \(m>p\). In particular, no asymptotic formula with a fixed
argument has been used.

We first consider subsequences on which \(p\le m/2\).  Extract a further
subsequence along which \(y=p/m\to y_0\in[0,1/2]\).  If \(y_0=0\), the
following dilute expansion is needed.  For \(r=1,\ldots,p-1\), the
fundamental theorem of calculus gives
\[
\psi_1\left(\frac{m-r}{2}\right)-\psi_1\left(\frac m2\right)
=\int_{(m-r)/2}^{m/2}-\psi_2(u)\dd u.
\]
For \(\ell=2\), \eqref{eq:app-polygamma-series} gives
\(-\psi_2(u)=2\sum_{k=0}^\infty(u+k)^{-3}\).
Integral comparison bounds this expression between \(u^{-2}\) and
\(u^{-2}+2u^{-3}\).  Hence, uniformly on the interval \(u\ge m/4\),
\(-\psi_2(u)=u^{-2}+O(u^{-3})\).
Summing the integrals,
\begin{equation}
v_{m,p}
=\frac{p(p-1)}{m^2}
+O\left(\frac{p^3}{m^3}\right).
\label{eq:app-sparse-v}
\end{equation}
The power series for \(-\log(1-y)\), with \(0\le y\le1/2\), gives
\[
\sigma_{0,m,p}^2
=2\{-\log(1-y)-y\}
=y^2+O(y^3),
\qquad
\sigma_{0,m,p}\ge y.
\]
Because \(d\ge m/2\), equations
\eqref{eq:app-coarse-equivalence} and \eqref{eq:app-sparse-v} give, when
\(y\to0\),
\[
\frac{|b_{m,p}-\mu_{0,m,p}|}{\sigma_{0,m,p}}
\le\frac Cp\longrightarrow0,
\qquad
\left|\frac{v_{m,p}}{\sigma_{0,m,p}^2}-1\right|
\le C\left(\frac1p+\frac pm\right)\longrightarrow0.
\]
If instead \(y_0\in(0,1/2]\), then
\[
\sigma_{0,m,p}^2
\longrightarrow2\{-\log(1-y_0)-y_0\}>0.
\]
Since \(d\ge m/2\to\infty\), division of
\eqref{eq:app-coarse-equivalence} by
\(\sigma_{0,m,p}\) and \(\sigma_{0,m,p}^2\) proves the two desired limits
on this further subsequence.  Thus every subsequence in the half
\(p\le m/2\) has a further subsequence on which both limits hold; the
subsequence principle proves them throughout this half. It remains to consider \(p>m/2\).  Start with an arbitrary subsequence.
If \(d\) is unbounded along it, extract a further subsequence on which
\(d\to\infty\).  Then
\[
\sigma_{0,m,p}^2
=2\{-\log(d/m)-(1-d/m)\}
\ge2\{\log(2)-1/2\}>0.
\]
Therefore \(d\sigma_{0,m,p}\to\infty\) and
\(d\sigma_{0,m,p}^2\to\infty\).  If \(d\) is bounded along the chosen
subsequence, extract a further subsequence on which the integer \(d\) is
constant.  On that further subsequence,
\eqref{eq:app-elementary-rewrite} and
\eqref{eq:app-coarse-equivalence} give
\[
\begin{aligned}
\sigma_{0,m,p}^2&=2\log(m)+O_d(1),\\
\frac{b_{m,p}-\mu_{0,m,p}}{\sigma_{0,m,p}}
&=O_d((\log(m))^{-1/2}),\\
\frac{v_{m,p}}{\sigma_{0,m,p}^2}-1
&=O_d((\log(m))^{-1}).
\end{aligned}
\]
In particular, the same two products diverge. Dividing the two inequalities in
\eqref{eq:app-coarse-equivalence} by \(\sigma_{0,m,p}\) and
\(\sigma_{0,m,p}^2\), respectively, proves the two desired limits on the
extracted further subsequence.  Thus every subsequence in the half
\(p>m/2\) has a further subsequence on which both limits hold.  The
subsequence principle, together with the preceding \(p\le m/2\) argument,
proves the result for every sequence with \(m>p\).
\end{proof}

\section*{Acknowledgments}
The author thanks Tuan Pham for helpful comments on an earlier version.
OpenAI's GPT-5.6 Sol Ultra was used to assist with literature
searches and organization, mathematical drafting and checks, language
editing, \LaTeX\ formatting, and computational checks. The author
independently verified and retains full responsibility for the
mathematical content, originality, citations, and text.

\bibliographystyle{plainnat}
\bibliography{new2026_R6_refs}

\end{document}